%% file: percolation2.tex
\documentclass[a4paper,12pt]{amsart}
\usepackage[hmargin=2.5cm,vmargin=2.5cm]{geometry}
\usepackage{hyperref}
\usepackage{graphicx}
\usepackage{amsmath}
\usepackage{amsthm}
\usepackage{amssymb}
\usepackage{bbm}
\usepackage{dsfont}
\usepackage{enumitem}
\usepackage{calligra}
\usepackage{tikz}
\usepackage{tkz-euclide}
\usepackage{etoolbox}
\usepackage{mathtools}

\usetikzlibrary{patterns,snakes}
\usetikzlibrary{decorations.pathreplacing,calc,positioning}
\hypersetup{colorlinks=true}
\usepackage{color}
\usepackage{accents}
\newlength{\dhatheight}
\newcommand{\doublehat}[1]{%
    \settoheight{\dhatheight}{\ensuremath{\hat{#1}}}%
    \addtolength{\dhatheight}{-0.35ex}%
    \hat{\vphantom{\rule{1pt}{\dhatheight}}%
    \smash{\hat{#1}}}}
\newcommand{\smalldoublehat}[1]{%
    \settoheight{\dhatheight}{\ensuremath{\hat{#1}}}%
    \addtolength{\dhatheight}{-0.95ex}%
    \hat{\vphantom{\rule{1pt}{\dhatheight}}%
    \smash{\hat{#1}}}}
\newcommand{\hhatSS}[1]{%
    \settoheight{\dhatheight}{\ensuremath{\scriptscriptstyle{\hat{#1}}}}%
    \addtolength{\dhatheight}{0.05ex}%
    \hat{\vphantom{\rule{1pt}{\dhatheight}}%
    \smash{\hat{#1}}}}

\newtheorem{thm}{Theorem}
\newtheorem{remark}[thm]{Remark}

\newtheorem{lemma}{Lemma}

\newtheorem{claim}{Claim}

\renewcommand{\P}{\mathbb P}

\allowdisplaybreaks

\makeatletter
\patchcmd{\@maketitle}
  {\ifx\@empty\@dedicatory}
  {\ifx\@empty\@date \else {\vskip3ex \centering\footnotesize\@date\par\vskip1ex}\fi
   \ifx\@empty\@dedicatory}
  {}{}
\patchcmd{\@adminfootnotes}
  {\ifx\@empty\@date\else \@footnotetext{\@setdate}\fi}
  {}{}{}
\makeatother

\begin{document}

\author{R\'eka Szab\'o and Daniel Valesin}
\address{University of Groningen, Nijenborgh 9, 9747 AG Groningen, The Netherlands}
\email{r.szabo@rug.nl, d.rodrigues.valesin@rug.nl}
\date{March 18, 2019}

\title{Inhomogeneous percolation on ladder graphs}

 \begin{abstract}
We define an inhomogeneous percolation model on ``ladder graphs''  obtained as direct products of an arbitrary graph~$G = (V,E)$ and the set of integers~$\mathbb{Z}$ (vertices are thought of as having a ``vertical'' component indexed by an integer). We make two natural choices for the set of edges, producing an unoriented graph~$\mathbb{G}$ and an oriented graph~$\vec{\mathbb{G}}$.
These graphs are endowed with percolation configurations in which independently, edges inside a fixed infinite ``column'' are open with probability~$q$, and all other edges are open with probability~$p$.
For all fixed~$q$ one can define the critical percolation threshold~$p_c(q)$. We show that this function is continuous in~$(0, 1)$.
 \end{abstract}

\subjclass[2010]{MSC 60K35, MSC 82B43}
\keywords{Inhomogeneous percolation, Oriented percolation, Ladder graphs, Critical parameter}

\maketitle

\section{Introduction}
\label{sec:intro}

In this paper we examine how the critical parameter of percolation is affected by inhomogeneities. More specifically, we address the following problem.  Suppose~$\mathbb{G}$ is a graph with (oriented or unoriented) set of edges~$\mathbb{E}$, and that~$\mathbb{E}$ is split into two disjoint sets,~$\mathbb{E} = \mathbb{E}' \cup \mathbb{E}''$. Consider the percolation model in which edges of~$\mathbb{E}'$ are open with probability~$p$ and edges of~$\mathbb{E}''$ are open with probability~$q$. For~$q \in [0,1]$, we can then define~$p_c(q)$ as the supremum of values of~$p$ for which percolation does not occur at~$p,q$. What can be said about the function~$q \mapsto p_c(q)$?

This is the framework for the problem of interest of the recent reference~\cite{LRV17}. In that paper, the authors consider an oriented tree whose vertex set is that of the~$d$-regular, rooted tree, and containing ``short edges'' (with which each vertex points to its~$d$ children) and ``long edges'' (with which each vertex points to its~$d^k$ descendants at distance~$k$, for fixed~$k \in \mathbb{N}$). Percolation is defined on this graph by letting short edges be open with probability~$p$ and long edges with probability~$q$. It is proved that the curve~$q \mapsto p_c(q)$ is continuous and strictly decreasing in the region where it is positive.

In the present paper, we consider another natural setting for the problem described in the first paragraph, namely that of a ``ladder graph'' in the spirit of~\cite{GN90}.
We start with an arbitrary (unoriented, connected) graph~$G=(V,E)$ and construct~$\mathbb{G}=(\mathbb{V},\mathbb{E})$ by placing layers of~$G$ one on top of the other and adding extra edges to connect the consecutive layers.
More precisely,~$\mathbb{V} = V \times \mathbb{Z}$ and~$\mathbb{E}$ consists of the edges that make each individual layer a copy of~$G$, as well as edges linking each vertex to its copies in the layers above it and below it (see Figure~\ref{fig:genericex} for an example).
With this choice (and other ones we will also consider), one would expect the aforementioned function~$p_c(q)$ to be constant in~$(0,1)$. Our main result is that it is a continuous function. We also consider a similarly defined oriented model~$\vec{\mathbb{G}}$, and obtain the same result. See Section~\ref{ss:formal_intro} for a more formal description of the models we study and the results we obtain.

\begin{figure}[ht]
 \begin{center}
 \def\svgwidth{450pt}
 \input{graphs.eps_tex}\qquad
 \caption{The construction of~$\mathbb{G}$ from~$G$ and a possible choice for the edge set~$\mathbb{E}''$ (on which edges are open with probability~$q$).}\label{fig:genericex}
 \end{center}
 \end{figure}
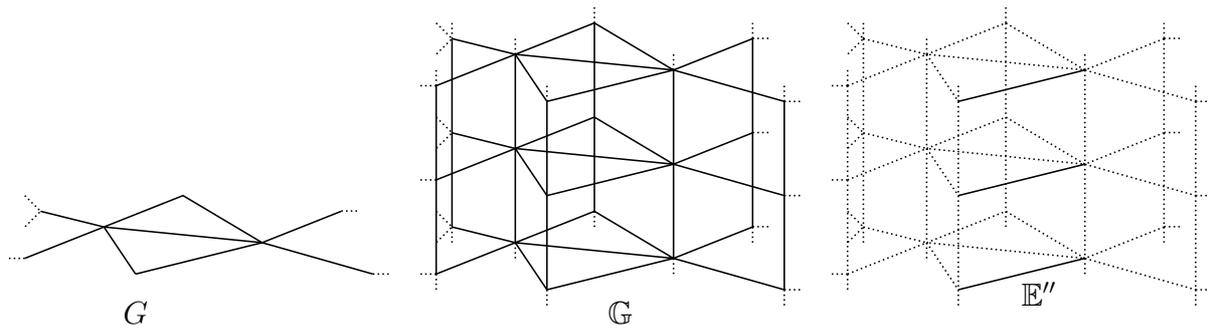
 
Our ladder graph percolation model is a generalization of the model of~\cite{Z94}. 
In that paper, Zhang considers an independent bond percolation model on~$\mathbb{Z}^2$ in which edges belonging to the vertical line through the origin are open with probability~$q$,
while other edges are open with probability~$p$.  It then follows from standard results in Percolation Theory that~$(0,1) \ni q \mapsto p_c(q)$ is constant, equal to~$\frac 1 2$, the critical value of (homogeneous) bond percolation on~$\mathbb{Z}^2$. The main result of~\cite{Z94} is that, when~$p$ is set to this critical value and for any~$q \in (0,1)$, there is almost surely no infinite percolation cluster. Since we are far from understanding the critical behaviour of homogeneous percolation on the more general graphs~$\mathbb{G}$ and~$\vec{\mathbb{G}}$ we consider here, analogous results to that of Zhang are beyond the scope of our work.

Let us  briefly mention some other related works. Important references for percolation phase transition beyond~$\mathbb{Z}^d$ are~\cite{BS96} and~\cite{LP16}; see also~\cite{CT17} for a recent development. Concerning sensitivity of the percolation threshold to an extra parameter or inhomogeniety of the underlying model, see the theory of essential enhancements developed in~\cite{AG91} and~\cite{BBR14}.

\subsection{Formal description of model and results}
\label{ss:formal_intro}
Let~$G=(V, E)$ be a connected graph with vertex set~$V$ and edge set~$E$.
Let~$\mathbb{V}=V\times \mathbb{Z}$. We define the unoriented graph~$\mathbb{G} = (\mathbb{V}, \mathbb{E})$ and the oriented graph~$\vec{\mathbb{G}} = (\mathbb{V}, \vec{\mathbb{E}})$, where 
\begin{align*}
\mathbb{E}=&\{\{(u,n), (v, n)\}: \{u, v\}\in E, n\in\mathbb{Z}\}\cup\{\{(u,n), (u, n+1)\}: u \in V, n\in\mathbb{Z}\},\\
\vec{\mathbb{E}}=&\{\langle(u,n), (v, n+1)\rangle: \{u, v\}\in E, n\in\mathbb{Z}\};
\end{align*}
above we denote unoriented edges by~$\{\cdot, \cdot\}$ and oriented edges by~$\langle\cdot, \cdot\rangle$.
See Figure~\ref{fig:graphs} for an example. Note that~$\vec{\mathbb{G}}$ is not necessarily connected.

\begin{figure}[ht]
\centering
 \begin{tikzpicture}

 \draw (1,0)  -- (4,0) ;
 \draw [dotted] (0.75,0)  -- (4.25,0) ;
 \draw (1,0.5)  -- (4,0.5) ;
 \draw [dotted] (0.75,0.5)  -- (4.25,0.5) ;
 \draw (1,1)  -- (4,1) ;
 \draw [dotted] (0.75,1)  -- (4.25,1) ;
 \draw (1,1.5)  -- (4,1.5) ;
 \draw [dotted] (0.75,1.5)  -- (4.25,1.5) ;
  \draw (1,2)  -- (4,2);
 \draw [dotted] (0.75,2)  -- (4.25,2) ;
 \draw (1,2.5)  -- (4,2.5) ;
 \draw [dotted] (0.75,2.5)  -- (4.25,2.5) ;
 \draw (1,3)  -- (4,3) ;
 \draw [dotted] (0.75,3)  -- (4.25,3) ;
 
 \draw (1,0)  -- (1,3) ;
 \draw [dotted] (1,-0.25)  -- (1,3.25) ;
  \draw (1.5,0)  -- (1.5,3) ;
 \draw [dotted] (1.5,-0.25)  -- (1.5,3.25) ;
  \draw (2,0)  -- (2,3) ;
 \draw [dotted] (2,-0.25)  -- (2,3.25) ;
  \draw (2.5,0)  -- (2.5,3) ;
 \draw [dotted] (2.5,-0.25)  -- (2.5,3.25) ;
   \draw (3,0)  -- (3,3) ;
 \draw [dotted] (3,-0.25)  -- (3,3.25) ;
  \draw (3.5,0)  -- (3.5,3) ;
 \draw [dotted] (3.5,-0.25)  -- (3.5,3.25) ;
  \draw (4,0)  -- (4,3) ;
 \draw [dotted] (4,-0.25)  -- (4,3.25) ;

\draw (2.5,-0.5)  node [below] {$\mathbb{G}$};

\draw (9,-0.5)  node [below] {$\vec{\mathbb{G}}$};

\foreach \n in {7.5,8.5,9.5}
	\foreach \m in {0,1,2}
		\draw [->] (\n,\m)-- ++(0.5,0.5);
\foreach \n in {7.5,8.5,9.5}
	\foreach \m in {1,2,3}
		\draw [<-] (\n,\m)-- ++(0.5,-0.5);
\foreach \n in {1,2,3}
	\foreach \m in {0,1, 2}
		\draw [->] (\n+7.5,\m)-- ++(-0.5,0.5);
\foreach \n in {1,2,3}
	\foreach \m in {1,2,3}
		\draw [<-] (\n+7.5,\m)-- ++(-0.5,-0.5);
		
\foreach \n in {7.5,8.5,9.5}
	\foreach \m in {0,1,2}
		\draw [dotted, ->] (\n+0.5,\m)-- ++(0.5,0.5);
\foreach \n in {7.5,8.5,9.5}
	\foreach \m in {1,2,3}
		\draw [dotted,<-] (\n+0.5,\m)-- ++(0.5,-0.5);
\foreach \n in {0, 1,2}
	\foreach \m in {0,1, 2}
		\draw [dotted,->] (\n+8,\m)-- ++(-0.5,0.5);
\foreach \n in {0,1,2}
	\foreach \m in {1,2,3}
		\draw [dotted,<-] (\n+8,\m)-- ++(-0.5,-0.5);

\end{tikzpicture}
 \caption{$\mathbb{G}$ and $\vec{\mathbb{G}}$ for $G=\mathbb{Z}$. Note that in this case, $\vec{\mathbb{G}}$ consists of two disjoint subgraphs; for clarity we will only display one of these subgraphs further on.} \label{fig:graphs}
\end{figure}
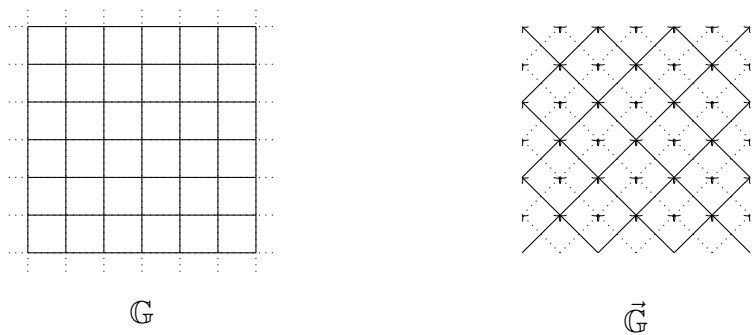

We consider percolation configurations in which each edge in~$\mathbb{E}$ and~$\vec{\mathbb{E}}$ can be \textit{open} or \textit{closed}.
Let~$\Omega=\{0, 1\}^{\mathbb{E}}$ and~$\vec{\Omega}=\{0, 1\}^{\vec{\mathbb{E}}}$ be the sets of all possible configurations on~$\mathbb{G}$ and~$\vec{\mathbb{G}}$, respectively. 
Then for any~$e \in \mathbb{E}$ or~$\vec{\mathbb{E}}$,~$\omega(e)=1$ corresponds to the edge being open and~$\omega(e)=0$ to closed.

An \textit{open path} on~$\mathbb{G}$ is a set of distinct vertices~$(v_0, n_0), (v_1, n_1), \dots, (v_m, n_m)$ such that for every~$i=0, \dots, m-1$,~$\{(v_i, n_i),(v_{i+1}, n_{i+1})\}\in \mathbb{E}$ and is open.
We say that~$(v, n)$ can be reached from~$(v_0, n_0)$ either if they are equal or if there is an open path from~$(v_0, n_0)$ to~$(v, n)$. Denote this event by~$(v_0, n_0)\leftrightarrow(v, n)$.
The set of vertices that can be reached from~$(v, n)$ is called the \textit{cluster} of~$(v, n)$.

An \textit{open path} on~$\vec{\mathbb{G}}$ can be defined similarly, but since edges are oriented upwards,~$(v, n)$ can only be reached from~$(v_0, n_0)$ if~$n\geq n_0$.
Denote this event by~$(v_0, n_0)\rightarrow(v, n)$.
Hence we will call the set of vertices that can be reached by an open path from~$(v, n)$ the \textit{forward cluster} of~$(v, n)$.
Denote by~$C_{\infty}$ and~$\vec{C}_{\infty}$ the events that there is an infinite cluster on~$\mathbb{G}$ and an infinite forward cluster on~$\vec{\mathbb{G}}$ respectively.

We examine the following inhomogeneous percolation setting.
First consider the unoriented graph~$\mathbb{G}$. Fix finitely many edges and vertices 
\begin{equation} \label{eq:set_of_edges} e_1=\{u_1, v_1\}, \dots, e_K=\{u_K, v_K\}\in E, \quad w_1, \dots w_L\in V\end{equation} and let 
\begin{align}
&\mathbb{E}^i:=\{\{(u_i,n), (v_i, n)\}: n\in\mathbb{Z}\} \hspace{20pt} i=1, \dots, K; \label{eq:einonoriented1}\\
&\mathbb{E}^{K+j}:=\{\{(w_j,n), (w_j, n+1)\}: n\in\mathbb{Z}\} \hspace{20pt} j=1, \dots, L; \label{eq:einonoriented2}
\end{align}
that is the set of ``horizontal'' edges on~$\mathbb{G}$ between~$u_i$ and~$v_i$, and the set of ``vertical'' edges above and below vertex~$w_j$ respectively (see Figure~\ref{fig:edgesets} for an example).
Further let~$\bold{q}=(q_1, \dots, q_{K+L})$ with~$q_i\in(0,1)$ for all~$i$ and let~$p\in[0, 1]$.
Now let each edge of~$\mathbb{E}^i$ be open with probability~$q_i$, and each edge in~$\mathbb{E}\setminus\cup_{i=1}^{K+L}\mathbb{E}^i$ be open with probability~$p$.
Denote the law of the open edges by~$\mathbb{P}_{\bold{q}, p}$.
Whether or not the event~$C_{\infty}$ happens with positive probability depends on the parameters~$p$ and~$\bold{q}$, so we can define the \textit{critical parameter} as a function of~$\bold{q}$:
\[
p_c(\bold{q}):=\sup\{p:\mathbb{P}_{\bold{q}, p}(C_{\infty})=0\}.                                                                                
\]
We will show that this function is continuous:

\begin{thm}\label{thm:percmulti}
For fixed~$K, L\in\mathbb{N}$, the function~$\bold{q}\mapsto p_c(\bold{q})$ is continuous in~$(0,1)^{K+L}$.
\end{thm}

\begin{figure}[ht]
\centering
 \begin{tikzpicture}

 \draw (2,0)  -- (2.5,0) ;
 \draw [dotted] (0.75,0)  -- (4.25,0) ;
 \draw (2,0.5)  -- (2.5,0.5) ;
 \draw [dotted] (0.75,0.5)  -- (4.25,0.5) ;
 \draw (2,1)  -- (2.5,1) ;
 \draw [dotted] (0.75,1)  -- (4.25,1) ;
 \draw (2,1.5)  -- (2.5,1.5) ;
 \draw [dotted] (0.75,1.5)  -- (4.25,1.5) ;
  \draw (2,2)  -- (2.5,2);
 \draw [dotted] (0.75,2)  -- (4.25,2) ;
 \draw (2,2.5)  -- (2.5,2.5) ;
 \draw [dotted] (0.75,2.5)  -- (4.25,2.5) ;
 \draw (2,3)  -- (2.5,3) ;
 \draw [dotted] (0.75,3)  -- (4.25,3) ;
 
 \draw [dotted] (1,-0.25)  -- (1,3.25) ;
 \draw [dotted] (1.5,-0.25)  -- (1.5,3.25) ;
 \draw [dotted] (2,-0.25)  -- (2,3.25) ;
 \draw [dotted] (2.5,-0.25)  -- (2.5,3.25) ;
   \draw (3,0)  -- (3,3) ;
 \draw [dotted] (3,-0.25)  -- (3,3.25) ;
 \draw [dotted] (3.5,-0.25)  -- (3.5,3.25) ;
 \draw [dotted] (4,-0.25)  -- (4,3.25) ;

\draw (2.5,-0.5)  node [below] {$\mathbb{G}$};
\draw (2.25,3.25)  node [above] {$\mathbb{E}^1$};
\draw (3,3.25)  node [above] {$\mathbb{E}^2$};

 \foreach \n in {2.5,9}
	\draw (\n,-0.1)  -- ++(0,0.2) ;
\draw (2.5,-0.1)  node [below] {0};
\draw (9,-0.1)  node [below] {0};
 
\draw (9,-0.5)  node [below] {$\vec{\mathbb{G}}$};
\draw (8.75,3.25)  node [above] {$\mathbb{E}^1$};
\draw (9.75,3.25)  node [above] {$\mathbb{E}^2$};

\foreach \n in {7.5,8.5,9.5}
	\foreach \m in {0,1,2}
		\draw [dotted, ->] (\n,\m)-- ++(0.5,0.5);
\foreach \n in {7.5,8.5,9.5}
	\foreach \m in {1,2,3}
		\draw [dotted, <-] (\n,\m)-- ++(0.5,-0.5);
\foreach \n in {1,2,3}
	\foreach \m in {0,1, 2}
		\draw [dotted, ->] (\n+7.5,\m)-- ++(-0.5,0.5);
\foreach \n in {1,2,3}
	\foreach \m in {1,2,3}
		\draw [dotted, <-] (\n+7.5,\m)-- ++(-0.5,-0.5);
		
\foreach \n in {7.5,8.5,9.5}
	\foreach \m in {0,1,2}
		\draw [dotted, ->] (\n+0.5,\m)-- ++(0.5,0.5);
\foreach \n in {7.5,8.5,9.5}
	\foreach \m in {1,2,3}
		\draw [dotted,<-] (\n+0.5,\m)-- ++(0.5,-0.5);
\foreach \n in {0, 1,2}
	\foreach \m in {0,1, 2}
		\draw [dotted,->] (\n+8,\m)-- ++(-0.5,0.5);
\foreach \n in {0,1,2}
	\foreach \m in {1,2,3}
		\draw [dotted,<-] (\n+8,\m)-- ++(-0.5,-0.5);

\foreach \n in {8.5,9.5}
	\foreach \m in {0,1,2}
		\draw [->] (\n,\m)-- ++(0.5,0.5);
\foreach \n in {8.5,9.5}
	\foreach \m in {1,2,3}
		\draw [<-] (\n,\m)-- ++(0.5,-0.5);
\foreach \n in {8.5,9.5}
	\foreach \m in {0,1,2}
		\draw [->] (\n+0.5,\m)-- ++(-0.5,0.5);
\foreach \n in {8.5,9.5}
	\foreach \m in {1,2,3}
		\draw [->] (\n,\m-0.5)-- ++(0.5,0.5);

\end{tikzpicture}
 \caption{The edge sets $\mathbb{E}^1$ and $\mathbb{E}^2$ on $\mathbb{G}$ with~$e_1=\{-1, 0\}$ and $w_1=1$; and on $\vec{\mathbb{G}}$ with~$e_1=\{-1, 0\}$ and $e_2=\{1, 2\}$ (for $G=\mathbb{Z}$).} \label{fig:edgesets}
\end{figure}
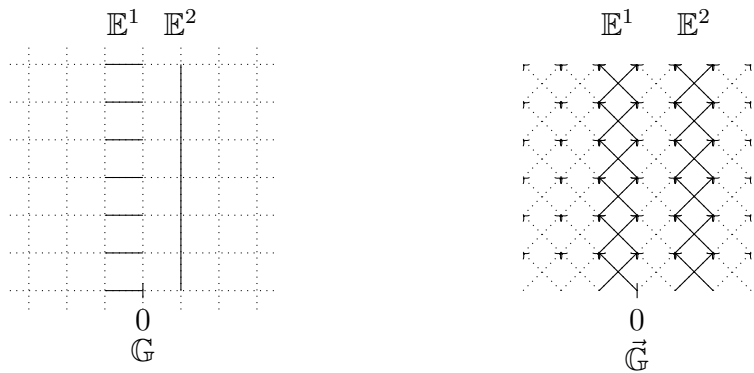

We now turn to the oriented graph~$\vec{\mathbb{G}}$.
Fix finitely many edges 
\begin{equation} \label{eq:edges2}
e_1=\{u_1, v_1\}, \dots, e_K=\{u_K, v_K\}\in E
\end{equation}
 and let 
\begin{equation}\label{eq:eioriented}
\vec{\mathbb{E}}^i:=\{\langle(u_i,n), (v_i, n+1)\rangle, \langle(v_i,n), (u_i, n+1)\rangle: n\in\mathbb{Z}\};
\end{equation}
that is the set of oriented edges on~$\vec{\mathbb{G}}$ between~$u_i$ and~$v_i$ (see Figure~\ref{fig:edgesets} for an example).
Further let~$\bold{q}=(q_1, \dots, q_K)$ with~$q_i\in(0,1)$ for all~$i$ and let~$p\in[0, 1]$.
Now let each oriented edge of~$\vec{\mathbb{E}}^i$ be open with probability~$q_i$, and each oriented edge in~$\vec{\mathbb{E}}\setminus\cup_{i=1}^K\vec{\mathbb{E}}^i$ be open with probability~$p$.
Denote the law of the open edges by~$\vec{\mathbb{P}}_{\bold{q}, p}$.
Similarly as in the unoriented case we can define the critical parameter as a function of~$\bold{q}$:
\[
\vec{p_c}(\bold{q}):=\sup\{p:\vec{\mathbb{P}}_{\bold{q}, p}(\vec{C}_{\infty})=0\}.                                                                                
\]
We will show that this function is continuous:

\begin{thm}\label{thm:orpercmulti}
For fixed~$K\in\mathbb{N}$, the function~$\bold{q}\mapsto \vec{p_c}(\bold{q})$ is continuous in~$(0,1)^K$.
\end{thm}
The proofs of both Theorem~\ref{thm:percmulti} and Theorem~\ref{thm:orpercmulti} rely on two coupling results which allow us to compare percolation configurations with different parameters~$\bold{q}, p$. These coupling results are presented in Section~\ref{s:prelim}. We prove Theorem~\ref{thm:percmulti} in Section~\ref{s:proof1}  and Theorem~\ref{thm:orpercmulti} in Section~\ref{s:proof2}.

\subsection{Discussion on the contact process} \label{ss:contact} 
Bond percolation on  the oriented graph~$\vec{\mathbb{G}}$ defined from~$G = (V,E)$ is closely related to the contact process on~$G$.
The contact process is usually taken as a model of epidemics on a graph: vertices are individuals, which can be healthy or infected.
In the continuous-time Markov dynamics infected individuals recover with rate~1 and transmit the infection to each neighbor with rate~$\lambda > 0$ (``infection rate'').
The ``all healthy'' configuration is a trap state for the dynamics; the probability that the contact process ever reaches this state is either equal to~1 or strictly less than~1 for any
finite set of initially infected vertices. The process is said to die out in the first case and
to survive in the latter. Whether it survives or dies out will depend on both the
underlying graph~$G$ and~$\lambda$, so one defines the critical rate~$\lambda_c$ as the supremum of the
infection parameter values for which the contact process dies out on~$G$. For a detailed introduction see~\cite{L13}.

The contact process admits a well-known graphical construction that is a ``space-time picture''~$G\times[0, \infty)$ of the process. We assign to each vertex~$v\in V$ and ordered pair of vertices~$(u, v)$ satisfying~$\{u, v\}\in E$ a Poisson point process~$D_v$ with rate~1 and~$D_{(u, v)}$ with rate~$\lambda$ respectively (all processes are independent). For each event time~$t$ of~$D_v$ we place a ``recovery mark'' at~$(v, t)$ and for each event time of~$D_{(u, v)}$ an ``infection arrow'' from~$(u, t)$ to~$(v, t)$. 
An ``infection path'' is a connected path that moves along the timeline in the increasing time direction, without passing through a recovery mark and along infection arrows in the direction of the arrow.
Starting from a set of initially infected vertices~$A\subset V$, the set of infected vertices at time~$t$ is the set of vertices~$v$ such that~$(v, t)$ can be reached by an infection path from some~$(u, 0)$ with~$u\in A$.

This representation can be thought of as a version of our oriented percolation model~$\vec{\mathbb{G}}$ in which the ``vertical'', one-dimensional component is taken as~$\mathbb{R}$ rather than~$\mathbb{Z}$ (some other modifications have to be made to account for the ``recovery marks'' of the contact process, but this is unimportant for the present discussion). In fact, one of the questions that originally motivated us was the following. Assume we take the contact process on an arbitrary graph~$G$, and declare that the infection rate is equal to~$\lambda > 0$ in every edge except for a distinguished edge~$e^*$, in which the infection rate is~$\sigma > 0$. Let~$\lambda_c(\sigma)$ be the supremum of values of~$\lambda$ for which the process with parameters~$\lambda, \sigma$ dies out (starting from finitely many infections). Is it true that~$\lambda_c(\sigma)$ is constant, or at least continuous, in~$(0,\infty)$?

In case~$G$ is a vertex-transitive connected graph, one can show that~$\lambda_c(\sigma)$ is constant in~$(0,\infty)$ by an argument similar to the one given in~\cite{J05}. For general~$G$, even continuity of~$\lambda_c(\sigma)$ is unproved, and the techniques we use here do not seem to be sufficient to handle that case
(see Remark~\ref{rem:contact_process} below for an explanation of what goes wrong).
 This is surprising, since results for oriented percolation typically transfer automatically to the contact process (and vice-versa).  A recent result shows that the situation can be quite delicate: in~\cite{SV16}, we exhibited a tree in which the contact process (with same rate~$\lambda > 0$ everywhere) survives for any value of~$\lambda$, but in which the removal of a single edge produces two subtrees in which the process dies out for small~$\lambda$.

\section{Coupling lemmas}\label{s:prelim}

The proofs of both our theorems are based on couplings which allow us to carefully compare percolation configurations sampled from measures with different parameter values.
In the proof of Theorem~\ref{thm:percmulti} we use the following coupling lemma (Lemma 3.1 from~\cite{LRV17}). The proof is omitted since it is quite simple and can be found in~\cite{LRV17};
the idea of the coupling is reminiscent of Doeblin's maximal coupling lemma (see \cite{T00} Chapter 1.4).
\begin{lemma}\label{lemma:coupling1}
 Let~$\mathbb{P}_{\theta}$ denote probability measures on a finite set~$S$, parametrized by $\theta\in(0, 1)^N$, such that~$\theta\mapsto \mathbb{P}_{\theta}(x)$ is continuous for every~$x\in S$.
 Assume that for some~$\theta_1$ and~$\bar{x} \in S$ we have~$\mathbb{P}_{\theta_1}(\bar{x})>0$.
 Then, for any~$\theta_2$ close enough to~$\theta_1$, there exists a coupling of two random elements~$X$ and~$Y$ of~$S$ such that 
~$X\sim\mathbb{P}_{\theta_1}$,~$Y\sim\mathbb{P}_{\theta_2}$ and 
\begin{equation}\label{eq:coupling1}
 \P\left(  \{X=Y\}\cup\{X=\bar{x}\}\cup\{Y=\bar{x}\}\right)=1.
\end{equation}
\end{lemma}
The following is a modified version of Lemma~\ref{lemma:coupling1}, to be used in the proof of Theorem~\ref{thm:orpercmulti}.
\begin{lemma}\label{lemma:coupling2}
 Let~$\mathbb{P}_{\theta}$ denote probability measures on a finite set~$S$, parametrized by $\theta\in(0, 1)^N$, such that~$\theta\mapsto \mathbb{P}_{\theta}(x)$ is continuous for every~$x\in S$.
 Let~$\{\hat{S}, \doublehat{S}\}$ be a non-trivial partition of~$S$, and assume that for some~$\theta_1$, $\hat{x}\in \hat{S}$ and~$\doublehat{x} \in \doublehat{S}$ we have~$\mathbb{P}_{\theta_1}(\hat{x})>0$ and~$\mathbb{P}_{\theta_1}(\doublehat{x})>0$.
 Then, for any~$\theta_2$ close enough to~$\theta_1$, there exists a coupling of two random elements~$X$ and~$Y$ of~$S$ such that 
~$X\sim\mathbb{P}_{\theta_1}$,~$Y\sim\mathbb{P}_{\theta_2}$ and 
\begin{equation}\label{eq:coupling2}
 \P\left(\{X=Y\}\cup\{X=\hat{x}\}\cup\{X\in \hat{S}\cup\{\doublehat{x}\}, Y=\hat{x}\} \cup\{Y=\doublehat{x}\} \right)=1,
\end{equation}
specifically
\begin{equation}\label{eq:coupling2spec}
  \P(Y=\hat{x} \text{ or } \doublehat{x} | X=\doublehat{x})1.
\end{equation}
\end{lemma}

\begin{proof}
We write~$\hat{S}=\{w_1, w_2, \dots, w_n, \hat{x}\}$ and~$\doublehat{S}=\{z_1, z_2, \dots, z_m, \doublehat{x}\}$ and for all~$y\in S$ and ~$k=1, 2$ let
\[
 \begin{array}{rclrcl}
 p(y)&=&\mathbb{P}_{\theta_1}(y)\wedge \mathbb{P}_{\theta_2}(y),&&& \\
 p_{\theta_1}(y)&=&[\mathbb{P}_{\theta_1}(y) - \mathbb{P}_{\theta_2}(y)]^+,& \hspace{20pt} p_{\theta_k}(\hat{S})&=&\sum_{y\in \hat{S}\setminus\{\hat{x}\}}p_{\theta_k}(y),\\
 p_{\theta_2}(y)&=&[\mathbb{P}_{\theta_2}(y) - \mathbb{P}_{\theta_1}(y)]^+, & \hspace{20pt} p_{\theta_k}(\doublehat{S})&=&\sum_{y\in \smalldoublehat{S}\setminus\{\hhatSS{x}\}}p_{\theta_k}(y).\\
\end{array}
\]

Let~$U$ be a uniform random variable on~$[0, 1]$. The values of~$X$ and~$Y$ will be given as functions of~$U$.
Clearly
\begin{align*}
 \sum_{i=1}^n p(w_i) + \sum_{j=1}^m p(z_j) + \mathbb{P}_{\theta_k}(\hat{x}) + p_{\theta_k}(\hat{S}) + \mathbb{P}_{\theta_k}(\doublehat{x}) + p_{\theta_k}(\doublehat{S})=1,
\end{align*}
so we can cover the line segment~$[0,1]$ with disjoint intervals with lengths equal to the summands of the left-hand side of the above equality with either~$k=1$ or~2 (see Figure~\ref{fig:01}).
For any value of~$u$ we choose~$X$ and~$Y$ to be the element of~$S$ that corresponds to the interval~$u$ falls into in the first and second cover respectively.

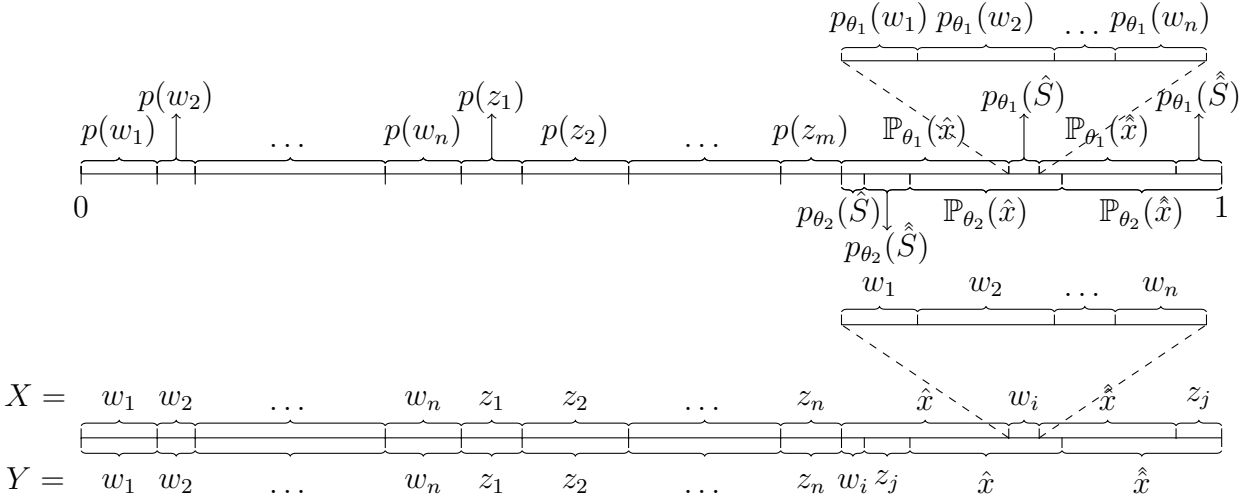
\begin{figure}[ht]
\centering
  \begin{tikzpicture}[every edge/.style={shorten <=1pt, shorten >=1pt}]
  \draw (0,0)   -- (15,0) ;
\draw (0,-0.15)  node [below] {0};
\draw (15,-0.15)  node [below] {1};

  \coordinate (p) at (0,4pt);
  \foreach \myprop/\mytext [count=\n] in {1/$p(w_1)$,0.5/,2.5/\dots,1/$p(w_n)$,0.8/,1.4/$p(z_2)$,2/\dots,0.8/$p(z_m)$,2.2/$\mathbb{P}_{\theta_1}(\hat{x})$,0.4/,1.8/$\mathbb{P}_{\theta_1}(\doublehat{x})$,0.6/}
  \draw [decorate,decoration={brace,amplitude=2}] (p)  edge [draw] +(0,-5pt) -- ++(\myprop,0) coordinate (p) node [midway, above=2 pt, anchor=south] {\mytext} ;
\coordinate (p) at (0,4pt);
  \foreach \myprop/\mytext [count=\n] in {1/,0.5/$p(w_2)$,2.5/,1/,0.8/$p(z_1)$,1.4/,2/,0.8/,2.2/,0.4/$p_{\theta_1}(\hat{S})$,1.8/,0.6/$p_{\theta_1}(\doublehat{S})$}
  \draw [decorate,decoration={brace,amplitude=2}] (p)  edge [draw] +(0,-5pt) -- ++(\myprop,0) coordinate (p) node [midway, above=15 pt, anchor=south] {\mytext} ;
  \path (15,4pt) edge [draw]  ++(0,-9pt);
 \coordinate (p) at (0,4pt);
\foreach \n in {1,0.5,2.5,1,0.8,1.4,2,0.8}
	\path (p)  edge [draw] +(0,-9pt) -- ++(\n,0)  coordinate (p);

 \draw [->](1.25,0.2) -- (1.25,0.8) ;
\draw [->](5.4,0.2) -- (5.4,0.8) ;
\draw [->](12.4,0.2) -- (12.4,0.8) ;
\draw [->](14.7,0.2) -- (14.7,0.8) ;
  
   \coordinate (p) at (10,-5pt);
  \foreach \myprop/\mytext [count=\n] in {0.3/$p_{\theta_2}(\hat{S}) \hspace{10pt}  $,0.6/,2/$\mathbb{P}_{\theta_2}(\hat{x})$,2.1/$\mathbb{P}_{\theta_2}(\doublehat{x})$}
  \draw [decorate,decoration={brace,amplitude=2, mirror}] (p)  edge [draw] +(0,6pt) -- ++(\myprop,0) coordinate (p) node [midway, below=20pt, anchor=south] {\mytext} ;
   \coordinate (p) at (10,-5pt);
  \foreach \myprop/\mytext [count=\n] in {0.3/,0.6/$p_{\theta_2}(\doublehat{S})$,2/,2.1/}
  \draw [decorate,decoration={brace,amplitude=2, mirror}] (p)  edge [draw] +(0,6pt) -- ++(\myprop,0) coordinate (p) node [midway, below=33pt, anchor=south] {\mytext} ;

\draw [->](10.6,-0.22) -- (10.6,-0.75) ;

 \draw [dashed](12.2,0) -- (10,1.5) ;
\draw [dashed](12.6,0) -- (14.8,1.5) ;
\draw (10,1.5) -- (14.8,1.5);
 \coordinate (p) at (10,1.64);
  \foreach \myprop/\mytext [count=\n] in {1/$p_{\theta_1}(w_1)$,1.8/$p_{\theta_1}(w_2)$,0.8/\dots,1.2/$p_{\theta_1}(w_n)$}
  \draw [decorate,decoration={brace,amplitude=2}] (p)  edge [draw] +(0,-5pt) -- ++(\myprop,0) coordinate (p) node [midway, above=2 pt, anchor=south] {\mytext} ;
 \path (14.8,1.64) edge [draw]  ++(0,-5pt);

\draw (0,-3.5)   -- (15,-3.5) ;
\coordinate (p) at (0,-3.34);
  \foreach \myprop/\mytext [count=\n] in {1/$w_1$,0.5/$w_2$,2.5/\dots,1/$w_n$,0.8/$z_1$,1.4/$z_2$,2/\dots,0.8/$z_n$,2.2/$\hat{x}$,0.4/$w_i$,1.8/$\doublehat{x}$,0.6/$z_j$}
  \draw [decorate,decoration={brace,amplitude=2}] (p)  edge [draw] +(0,-6pt) -- ++(\myprop,0) coordinate (p) node [midway, above=2pt, anchor=south] {\mytext} ;
  \path (15,0.5pt) edge [draw]  ++(0,-5pt);
  
   \coordinate (p) at (0,-3.67);
  \foreach \myprop/\mytext [count=\n] in {1/$w_1$,0.5/$w_2$,2.5/\dots,1/$w_n$,0.8/$z_1$,1.4/$z_2$,2/\dots,0.8/$z_n$,0.3/$w_i$,0.6/$z_j$,2/$\hat{x}$,2.1/$\doublehat{x}$}
  \draw [decorate,decoration={brace,amplitude=2, mirror}] (p)  edge [draw] +(0,6pt) -- ++(\myprop,0) coordinate (p) node [midway, below=20pt, anchor=south] {\mytext} ;

  \path (15,-3.34) edge [draw]  ++(0,-9pt);

\draw (-0.6,-2.65)  node [below] {$X=$};
\draw (-0.6,-3.75)  node [below] {$Y=$};

 \draw [dashed](12.2,-3.5) -- (10,-2) ;
\draw [dashed](12.6,-3.5) -- (14.8,-2) ;
\draw (10,-2) -- (14.8,-2);
 \coordinate (p) at (10,-1.83);
  \foreach \myprop/\mytext [count=\n] in {1/$w_1$,1.8/$w_2$,0.8/\dots,1.2/$w_n$}
  \draw [decorate,decoration={brace,amplitude=2}] (p)  edge [draw] +(0,-6pt) -- ++(\myprop,0) coordinate (p) node [midway, above=2 pt, anchor=south] {\mytext} ;
 \path (14.8,-1.83) edge [draw]  ++(0,-6pt);
  
\end{tikzpicture}
\caption{The partitioning of the line segment $[0,1]$, and the sampling of $(X, Y)$.} \label{fig:01}
\end{figure}

\pagebreak
To guarantee that \eqref{eq:coupling2} is satisfied we arrange these intervals in a way that 
\begin{itemize}
 \item the interval corresponding to~$\P_{\theta_1}(\doublehat{x})$ in the first cover is entirely contained in the intervals corresponding to~$\P_{\theta_2}(\hat{x})$ and~$\P_{\theta_2}(\doublehat{x})$ in the second cover;
 \item the interval corresponding to~$p_{\theta_1}(\doublehat{S})$ in the first cover is contained in the interval corresponding to~$\P_{\theta_2}(\doublehat{x})$ in the second cover;
 \item the interval corresponding to~$p_{\theta_1}(\hat{S})$ in the first cover is contained in the intervals corresponding to~$\P_{\theta_2}(\hat{x})$ and~$\P_{\theta_2}(\doublehat{x})$ in the second cover.
\end{itemize}

The above is possible since by continuity, as~$\theta_2\rightarrow\theta_1: \mathbb{P}_{\theta_2}(\hat{x})\rightarrow\mathbb{P}_{\theta_1}(\hat{x})>0$, $\mathbb{P}_{\theta_2}(\doublehat{x})\rightarrow\mathbb{P}_{\theta_1}(\doublehat{x})>0$ as well as~$p_{\theta_1}(\hat{S}), p_{\theta_1}(\doublehat{S})\rightarrow0$.
Therefore, if~$\theta_2$ is sufficiently close to~$\theta_1$, we have 
\begin{align*}
 p_{\theta_1}(\doublehat{S})&<\mathbb{P}_{\theta_2}(\doublehat{x}),\\
 p_{\theta_1}(\doublehat{S})+\mathbb{P}_{\theta_1}(\doublehat{x})+p_{\theta_1}(\hat{S})&<\mathbb{P}_{\theta_2}(\doublehat{x})+\mathbb{P}_{\theta_2}(\hat{x}).\\
\end{align*}
\end{proof}


\section{Proof of Theorem \ref{thm:percmulti}}
\label{s:proof1}
We start showing that if the statement of Theorem~\ref{thm:percmulti} is proved for a given set of edges and vertices as in~\eqref{eq:set_of_edges}, then the same continuity statement automatically follows for smaller sets of edges and vertices. To prove this, let $e_1,\ldots, e_K,\;w_1,\ldots, w_L$ be edges and vertices as in~\eqref{eq:set_of_edges}, and let~$w_{L+1}$ be an additional vertex (we could alternatively take an additional edge with no change to the argument that follows). We now compare two percolation models on~$\mathbb{G}$: the first one with parameter values~$\bold{q} = (q_1,\ldots, q_{K+L})$ for~$\mathbb{E}^1,\ldots, \mathbb{E}^{K+L}$ and ~$p$ for all other edges, and the second one with parameter values~$(\bold{q},q_{K+L+1})$ for~$\mathbb{E}^1,\ldots, \mathbb{E}^{K+L+1}$ and~$p$ for all other edges.

\begin{claim}\label{claim:contmm-1}
 If the function~$(\bold{q}, q_{K+L+1})\mapsto {p_c}(\bold{q}, q_{K+L+1})$ is continuous in~$(0,1)^{K+L+1}$, then~$\bold{q}\mapsto {p_c}(\bold{q})$ is continuous in~$(0,1)^{K+L}$.
 \end{claim}
 \begin{proof}
Since~$(0, 1)\ni q_{K+L+1}\mapsto {p_c}(\bold{q}, q_{K+L+1})$ is non-increasing and by assumption continuous, there exists a unique~$t^*\in(0, 1)$ such that~$t^*= {p_c}(\bold{q}, t^*)$.
We claim that~$t^*={p_c}(\bold{q})$.
Indeed, by the definition of~${p_c}(\bold{q}, t^*)$,
\begin{align*}
 &\forall t>t^*, \hspace{10pt} 0<{\mathbb{P}}_{(\bold{q}, t^*), t}(C_{\infty})\leq{\mathbb{P}}_{(\bold{q}, t), t}(C_{\infty})={\mathbb{P}}_{\bold{q}, t}(C_{\infty}), \text{ and}\\[.2cm]
 &\forall t<t^*, \hspace{10pt} 0={\mathbb{P}}_{(\bold{q}, t^*), t}(C_{\infty})\geq{\mathbb{P}}_{(\bold{q}, t), t}(C_{\infty})={\mathbb{P}}_{\bold{q}, t}(C_{\infty}),
\end{align*}
which implies~${p_c}(\bold{q})= t^*$.
 
 Assume that~${p_c}(\bold{q}, t)= t$ for some~$\bold{q}$ and~$t$.
 By continuity, for all~$\epsilon > 0$, if~$\boldsymbol{\delta} \in (0,1)^{K+L}$ is close enough to zero we have~$${p_c}(\bold{q}+\boldsymbol{\delta}, t)\in(t-\epsilon, t+\epsilon).$$
 As~${p_c}$ is non-increasing in~$t$, this yields
 \[
{p_c}(\bold{q}+\boldsymbol{\delta}, t-\epsilon)> t-\epsilon \hspace{10pt} \text{ and } \hspace{10pt} {p_c}(\bold{q}+\boldsymbol{\delta}, t+\epsilon)< t+\epsilon.
 \]
 Hence there exists~$t'\in (t-\epsilon, t+\epsilon)$ such that~${p_c}(\bold{q}+\boldsymbol{\delta}, t')=t'$.
 This implies that~$\bold{q}\mapsto {p_c}(\bold{q})$ is continuous.
\end{proof}

For our base graph~$G=(V,E)$,~$u, v \in V$ and~$V' \subset V$, let~$\text{dist}_G(u,v)$ be the graph distance between~$u$ and~$v$, and let~$\text{dist}_G(u,V')$ be the smallest graph distance between~$u$ and a point of~$V'$. Fix~$r \in \mathbb{N}$,~$u_0 \in V$ and let
\begin{equation}\label{eq:U}
U:=B_r(u_0), 
\end{equation}
that is the ball of radius~$r$ around~$u_0$ with respect to the graph distance.

\emph{From now on, we will assume that the edges~$e_1,\ldots, e_K$ of \eqref{eq:set_of_edges} are all the edges with both endpoints belonging to~$U$, and that the vertices~$w_1,\ldots, w_L$ of \eqref{eq:set_of_edges} are all the vertices of~$U$.} We are allowed to restrict ourselves to this case by Claim~\ref{claim:contmm-1}.

The proof of Theorem~\ref{thm:percmulti} will be a consequence of the following claim. 
\begin{claim}\label{claim:pqsurv}
 For all~$p\in(0,1)$,~$\bold{q_0}\in(0,1)^{K+L}$ and~$\epsilon\in(0, 1-p)$ there exists a~$\delta>0$ such that for any~$\bold{q}, \bold{q'}\in(0,1)^{K+L}$ satisfying~$\|\bold{q_0}-\bold{q}\|_{\infty}<\delta$ and~$\|\bold{q_0}-\bold{q'}\|_{\infty}<\delta$ we have
 \[
  \mathbb{P}_{\bold{q},\;p}(C_{\infty})\leq\mathbb{P}_{\bold{q'}, p+\epsilon}(C_{\infty}).
 \]
\end{claim}

Note that Claim~\ref{claim:pqsurv} is trivial if~$\bold{q'}-\bold{q}$ has non-negative coordinates.

 \begin{proof}[Proof of Theorem~\ref{thm:percmulti}]
 Fix~$\bold{q_0} \in (0,1)^{K+L}$ and~$\epsilon > 0$. By Claim~\ref{claim:pqsurv}, if~$\|\bold{q_0}-\bold{q'}\|_{\infty}$ is close enough to zero, then
\begin{align}\label{eq:aux_appl_claim1}
&\mathbb{P}_{\bold{q},\;p_c(\bold{q_0}) + \epsilon}(C_{\infty}) \geq \mathbb{P}_{\bold{q_0} ,\;p_c(\bold{q_0}) + \frac{\epsilon}{2}}(C_{\infty}),\\[.2cm]
\label{eq:aux_appl_claim2}&\mathbb{P}_{\bold{q},\;p_c(\bold{q_0}) - \epsilon}(C_{\infty}) \leq \mathbb{P}_{\bold{q_0} ,\;p_c(\bold{q_0}) - \frac{\epsilon}{2}}(C_{\infty}).
\end{align}
By the definition of~$p_c(\bold{q_0})$, the right-hand side of \eqref{eq:aux_appl_claim1} is positive and the right-hand side of \eqref{eq:aux_appl_claim2} is zero; hence, the two inequalities respectively yield 
$$p_c(\bold{q}) \leq p_c(\bold{q_0}) + \epsilon\quad \text{and}\quad p_c(\bold{q}) \geq p_c(\bold{q_0}) - \epsilon.$$
This implies that~$\bold{q} \mapsto p_c(\bold{q})$ is continuous at~$\bold{q_0}$.
\end{proof}

\begin{proof} [Proof of Claim~\ref{claim:pqsurv}]
We start with several definitions. Recall the definition of~$U$ in~\eqref{eq:U} and for~$n \in \mathbb{Z}$ let
$$\mathbb{V}_n = \{(v,m) \in \mathbb{V}:\; v\in~B_{r+1}(u_0),\; (2L+2)n \leq m \leq (2L+2)(n+1)\}$$
and
\begin{align*}
\mathbb{E}&_n = \{e \in \mathbb{E}:\; e \text{ has both endpoints in } \mathbb{V}_n\}\\[-.1cm]
&\backslash \{e \in \mathbb{E}: e = \{(u, (2L+2)(n+1)),(v, (2L+2)(n+1))\} \text{ for some } \{u,v\} \in E\}.
\end{align*}
We think of~$\mathbb{V}_n$ as a ``box'' of vertices and of~$\mathbb{E}_n$ as all the edges in the subgraph induced by this box, except for the ``ceiling''. Note that the~$\mathbb{E}_n$ are disjoint (though the~$\mathbb{V}_n$ are not). Next, recall the definition of~$\mathbb{E}^i$ for~$1 \leq i \leq K+L$ from \eqref{eq:einonoriented1} and \eqref{eq:einonoriented2}. Observe that~$\cup_i\mathbb{E}^i\subsetneq\cup_n\mathbb{E}_n$ and define, for~$n \in \mathbb{Z}$ and~$1 \leq i \leq K+L$,
$$\mathbb{E}^i_n = \mathbb{E}_n \cap \mathbb{E}^i,\qquad \mathbb{E}^\partial_n = \mathbb{E}_n \backslash \left(\cup_{i=1}^{K+L} \mathbb{E}^i_n\right),\qquad \mathbb{E}_\mathcal{O} = \mathbb{E} \backslash \left( \cup_{n \in \mathbb{Z}} \mathbb{E}_n\right).$$
The ``edge boundary''~$\mathbb{E}^\partial_n$ consists of edges of the form~$\{(u,m),(u,m+1)\}$, with~$u$ such that~$\text{dist}(u, u_0)=r+1$, and edges of the form~$\{(u,m),(v,m)\}$, with~$v \in U$ and~$\text{dist}(u, u_0)=r+1$. Next, let
$$\Omega_n^i = \{0,1\}^{\mathbb{E}_n^i},\quad \Omega_n^\partial = \{0,1\}^{\mathbb{E}_n^\partial}, \quad \Omega_n = \{0,1\}^{\mathbb{E}_n},\quad \Omega_\mathcal{O} = \{0,1\}^{\mathbb{E}_\mathcal{O}};$$ note that
$$\Omega = \Omega_\mathcal{O} \times \prod_{n\in \mathbb{Z}} \Omega_n = \Omega_\mathcal{O} \times \prod_{n\in \mathbb{Z}}\left( \Omega_n^\partial \times \prod_{i=1}^{K+L} \Omega_n^i\right).$$

For each~$n$, define the inner vertex boundary, consisting of the ``floor'', ``walls'' and ``ceiling'' of the vertex box~$\mathbb{V}_n$,
\begin{align*}\partial \mathbb{V}_n = \{(v,n)& \in \mathbb{V}_n: \text{dist}(v, u_0)=r+1\} \\ &\cup (U\times \{(2L+2)n\}) \cup (U \times \{(2L+2)(n+1)\}).\end{align*}
Given any~$\varnothing\neq A \subseteq \partial \mathbb{V}_n$ and~$\omega_n \in \Omega_n$, define
$$C_n(A,\omega_n) = \{(v,n) \in \partial \mathbb{V}_n: (v_0,n_0) \stackrel{\omega_n}{\longleftrightarrow} (v,n) \text{ for some }(v_0,n_0) \in A\},$$
where the notation~$(v_0,n_0)  \stackrel{\omega_n}{\longleftrightarrow} (v,n)$ means that~$(v_0,n_0)$ and~$(v,n)$ are connected by an~$\omega_n$-open path of edges of~$\mathbb{E}_n$. Note that~$A \subseteq C_n(A,\omega_n)$.

Now fix~$p$,~$\bold{q_0}$ and~$\epsilon$, and for~$\delta$ close enough to zero let~$\bold{q}=(q_1, \dots, q_{K+L})$ and~$\bold{q'}=(q'_1, \dots, q'_{K+L})$ be as in the statement of the claim. Note that~$\|\bold{q}-\bold{q'}\|_{\infty}<2\delta$.
We will define coupling measures~$\mu_\mathcal{O}$ on~$(\Omega_\mathcal{O})^2$ and~$\mu_n$ on~$(\Omega_n)^2$ satisfying the following properties. First, 
\begin{equation}\label{eq:prop_coup1}\begin{split}&(\omega_\mathcal{O},\omega'_\mathcal{O}) \sim \mu_\mathcal{O} \quad \Longrightarrow \quad \omega_\mathcal{O} \stackrel{\text{(d)}}{=} \mathbb{P}_{\bold{q}, p}\vert_{\mathbb{E}_\mathcal{O}},\quad \omega'_\mathcal{O} \stackrel{\text{(d)}}{=} \mathbb{P}_{\bold{q'}, p+\epsilon}\vert_{\mathbb{E}_\mathcal{O}} \\[.2cm]&\hspace{7cm}\text{ and } \omega_\mathcal{O} \leq \omega'_\mathcal{O} \text{ a.s.}\end{split}\end{equation}
(we denote by~$\mathbb{P}_{\bold{q}, p} \vert_{\mathbb{E}'}$ the projection of~$\mathbb{P}_{\bold{q}, p}$ to~$\mathbb{E}' \subset \mathbb{E}$).
Second, 
\begin{equation}\label{eq:prop_coup2}\begin{split}&(\omega_n,\omega'_n) \sim \mu_n \quad \Longrightarrow \quad \omega_n \stackrel{\text{(d)}}{=} \mathbb{P}_{\bold{q}, p}\vert_{\mathbb{E}_n},\quad \omega'_n \stackrel{\text{(d)}}{=} \mathbb{P}_{\bold{q'}, p+\epsilon}\vert_{\mathbb{E}_n} \\[.2cm]&\hspace{4cm}\text{ and } C_n(A,\omega_n) \subseteq C_n(A,\omega'_n) \text{ for all } A \subset \partial \mathbb{V}_n \text{ a.s.}\end{split}\end{equation}
We then define the coupling measure~$\mu$ on~$\Omega^2$ by
$$\mu = \mu_\mathcal{O} \otimes \left(\otimes_{n \in \mathbb{Z}} \mu_n \right).$$
It is clear from \eqref{eq:prop_coup1} and \eqref{eq:prop_coup2} that, if~$(\omega, \omega') \sim \mu$, then~$\omega \sim \mathbb{P}_{\bold{q}, p}$,~$\omega' \sim \mathbb{P}_{\bold{q'}, p+\epsilon}$, and almost surely if~$C_{\infty}$ holds for~$\omega$, then it holds for~$\omega'$. Consequently
\[
 \mathbb{P}_{\bold{q}, p}(C_{\infty})\leq\mathbb{P}_{\bold{q'}, p+\epsilon}(C_{\infty}).
\]

The definition of~$\mu_\mathcal{O}$ is standard. We take in some probability space a pair of random elements~$Z=(Z_1, Z_2) \in \Omega_\mathcal{O}^2$ such that~$Z_1$ and~$Z_2$ are independent on all edges of~$\mathbb{E}_\mathcal{O}$ and they assign each edge to be open with probability~$p$ and~$\frac{\epsilon}{1-p}$ respectively. We then let~$\omega_\mathcal{O}=Z_1$ and~$\omega'_O=Z_1\vee Z_2$, and~$\mu_\mathcal{O}$ be the distribution of~$(\omega_\mathcal{O},\omega'_\mathcal{O})$, so that \eqref{eq:prop_coup1} is clearly satisfied.

The measures~$\mu_n$ will be defined as translations of each other, so we only define~$\mu_0$. The construction relies on Lemma~\ref{lemma:coupling1}, with the finite set~$S$ of that lemma being here the set 
\[\Omega_0^1 \times \cdots \times \Omega^{K+L}_0 \times \Omega^\partial_0 \times \Omega^\partial_0.\]
The two factors of~$\Omega^\partial_0$ ensure the extra randomness needed for the coupling.
We now define the deterministic element~$\bar{x}$ of the above set that appears in the statement of Lemma~\ref{lemma:coupling1}. The definition is simple, but the notation is clumsy; a quick glimpse at Figure~\ref{fig:xbar} should clarify what is involved. We start assuming, without loss of generality, that the elements~$w_1, \dots, w_L$ of~$U$ are enumerated so that
\[
 \text{dist}_G(w_j, V\setminus U)\leq \text{dist}_G(w_{j+1}, V\setminus U) \hspace{20pt} \forall j=1, \dots, L-1.
\]

 Let~$\Gamma_j$ be the set of edges along a shortest path from~$w_j$ to~$U\setminus B_{r-1}(u_0)$.
Further for~$m<m'$ let
\[
[(w_i, m), (w_i, m')]:=\cup_{j=m}^{m'-1}\{(w_i, j), (w_i, j+1)\}.
\]

Now,~$\bar{x}$ is defined in the following way:
 \begin{itemize}
 \setlength\itemsep{0.3em}
  \item~$\bar{x}= (\bar{x}^{U}, \bar{x}^{\partial, 1}, \bar{x}^{\partial, 2})$ with~$\bar{x}^{U} \in \Omega_0^1 \times \dots \times \Omega_0^{K+L}$ and~$\bar{x}^{\partial, 1}, \bar{x}^{\partial, 2} \in \Omega_0^{\partial}$;
  \item~$\bar{x}^{U}(e)=1$ if and only if for some~$j=1, \dots L$,
  \begin{align*}
 e \in [(w_j, 0),& (w_j, j)] \cup [(w_j, (2L+2)-j), (w_j, (2L+2))] \\
  &\bigcup_{\{u, v\}\in \Gamma_j} \left(\{(u, j), (v, j)\} \cup \{(u, (2L+2)-j), (v, (2L+2)-j)\}\right),  
  \end{align*}
  or
  \[
  e\in \bigcup_{u, v\in U} \{(u, L+1), (v, L+1)\};
  \]
  \item~$\bar{x}^{\partial, 1}\equiv 0 \text{ and } \bar{x}^{\partial, 2}\equiv 1$.
 \end{itemize}

 \begin{figure}[ht]
\centering
 \begin{tikzpicture}
\draw [dotted] (0.5,0)  -- (3.5,0) ;


\foreach \n in {2,7,12}
    \foreach \m in {-0.1}
 	    \draw (\n,\m) node [below] {0};

\foreach \n in {2,7,12}
    \foreach \m in {-0.1}
 	    \draw (\n,\m)  -- ++(0,0.2) ;

 \draw (2,-0.5)  node [below] {$\bar{x}^U$};
\draw (7,-0.5)  node [below] {$\bar{x}^{\partial, 1}$};
\draw (12,-0.5)  node [below] {$\bar{x}^{\partial, 2}$};



  \foreach \n in {0.5,1,1.5,2,2.5,3,3.5}
 	    \draw [dotted](\n,0)  -- (\n,8) ;
  
  \foreach \m in {0,0.5,1,1.5,2,2.5,3,3.5,4,4.5,5,5.5,6,6.5,7,7.5}
 	    \draw [dotted](0.5,\m)  -- (3.5,\m) ;
  
  \draw  (0.5,0)  -- (0.5,0.5) ;
  \draw  (0.5,7.5)  -- (0.5,8) ;
  \draw  (1,0)  -- (1,1.5) ;
  \draw  (1,6.5)  -- (1,8) ;
  \draw  (1.5,0)  -- (1.5,2.5) ;
  \draw  (1.5,5.5)  -- (1.5,8) ;
  \draw  (2,0)  -- (2,3.5) ;
  \draw  (2,4.5)  -- (2,8) ;
  \draw  (2.5,0)  -- (2.5,3) ;
  \draw  (2.5,5)  -- (2.5,8) ;
  \draw  (3,0)  -- (3,2) ;
  \draw  (3,6)  -- (3,8) ;
  \draw  (3.5,0)  -- (3.5,1) ;
  \draw  (3.5,7)  -- (3.5,8) ;
  
  \draw  (0.5,1.5)  -- (1,1.5) ;
  \draw  (3,2)  -- (3.5,2) ;
  \draw  (0.5,2.5)  -- (1.5,2.5) ;
  \draw  (2.5,3)  -- (3.5,3) ;
  \draw  (0.5,3.5)  -- (2,3.5) ;
  \draw  (0.5,4)  -- (3.5,4) ;
  \draw  (0.5,4.5)  -- (2,4.5) ;
  \draw  (2.5,5)  -- (3.5,5) ;
  \draw  (0.5,5.5)  -- (1.5,5.5) ;
  \draw  (3,6)  -- (3.5,6) ;
  \draw  (0.5,6.5)  -- (1,6.5) ;

\draw [dotted] (5,0)  -- (5,8) ;
  \draw [dotted] (9,0)  -- (9,8) ;
  \foreach \n in {0,0.5,1,1.5,2,2.5,3,3.5,4,4.5,5,5.5,6,6.5,7,7.5}
	\foreach \m in {5,8.5}
		\draw [dotted] (\m,\n)-- ++(0.5,0);

 \draw  (10,0)  -- (10,8) ;
  \draw  (14,0)  -- (14,8) ;
  \foreach \n in {0,0.5,1,1.5,2,2.5,3,3.5,4,4.5,5,5.5,6,6.5,7,7.5}
	\foreach \m in {10,13.5}
		\draw  (\m,\n)-- ++(0.5,0);

\end{tikzpicture}
 \caption{The deterministic configuration for $G=\mathbb{Z}$, $U=\{-3, -2, -1, 0, 1, 2, 3\}$. In this case  $L=7$, $K=6$ and $w_1=-3, w_2=3, w_3=-2, w_4=2, w_5=-1, w_6=1, w_7=0$.} \label{fig:xbar}
\end{figure}
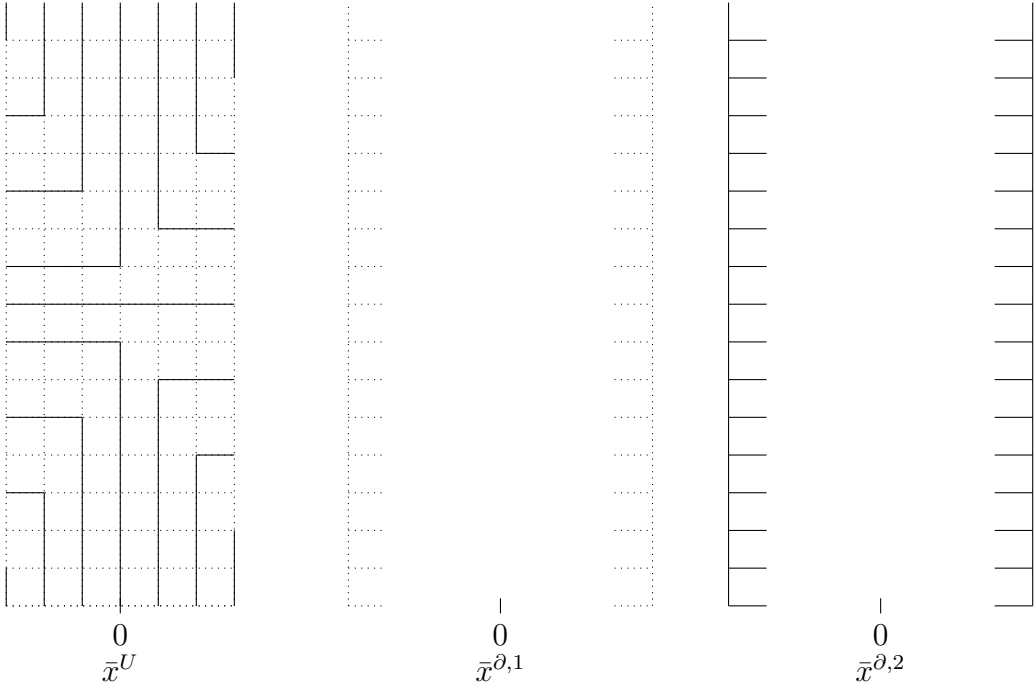
 
 By Lemma~\ref{lemma:coupling1}, if~$\delta$ is close enough to zero, then there exists a coupling of~$(K+L+2)$-tuples of configurations
 \begin{align*}
  X=(X^1, \dots, X^{K+L}, X^{\partial, 1}, X^{\partial, 2}),\; Y=(Y^1, \dots, Y^{K+L}, Y^{\partial, 1}, Y^{\partial, 2})\\
  \in \Omega^1_0 \times \dots \times \Omega^{K+L}_0 \times \Omega_0^{\partial} \times \Omega_0^{\partial}
  \end{align*}
such that
 \begin{itemize}
  \setlength\itemsep{0.3em}
  \item the values of~$X^1, \dots, X^{K+L}, X^{\partial, 1}, X^{\partial, 2}$ are independent on all edges;
  \item the values of~$Y^1, \dots, Y^{K+L}, Y^{\partial , 1}, Y^{\partial, 2}$ are independent on all edges;
  \item~$X^i$ assigns each edge to be open with probability~$q_i$;
  \item~$Y^i$ assigns each edge to be open with probability~$q'_i$;
  \item~$X^{\partial, 1}$ and~$Y^{\partial, 1}$ assign each edge to be open with probability~$p$;
  \item~$X^{\partial, 2}$ and~$Y^{\partial, 2}$ assign each edge to be open with probability~$\frac{\epsilon}{1-p}$;
  \item~$(X, Y)$ satisfies \begin{equation}\label{eq:speccoupling}
 \P\left(\{X=Y\}\cup\{X=\bar{x}\}\cup\{Y= \bar{x}\} \right)=1.
\end{equation}
 \end{itemize}
 Now let~$\omega_0=(X^1, \dots, X^{K+L}, X^{\partial, 1})$ and~$\omega_0'=(Y^1, \dots, Y^{K+L}, Y^{\partial, 1}\vee Y^{\partial, 2})$. 
 Thus~$\omega_0'$ assigns each edge in~$\mathbb{E}_0^{\partial}$ to be open with probability~$p+\epsilon$. 
 See Figure~\ref{fig:fixedconfigmulti} for~$\omega_0$ and~$\omega'_0$ if~$X$ or~$Y$ equals~$\bar{x}$.
 
 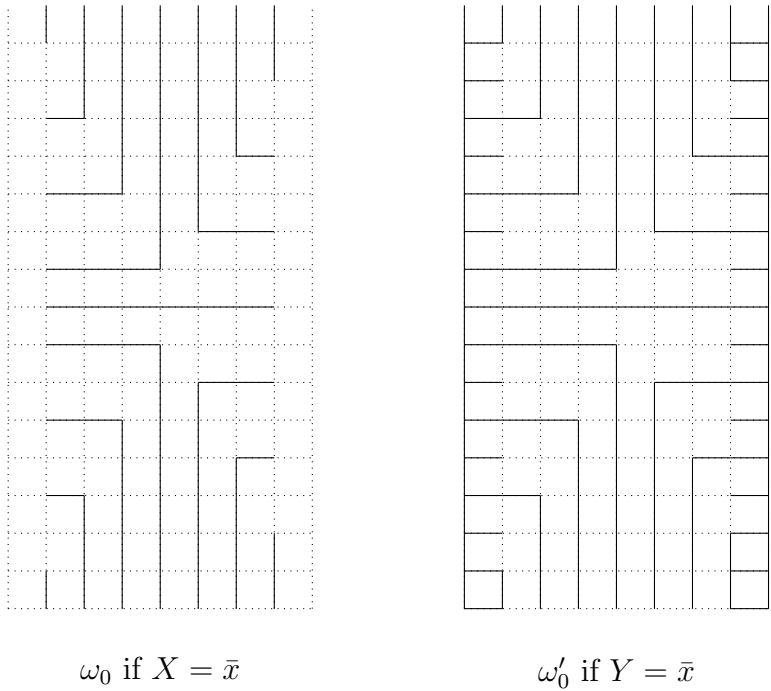
\begin{figure}[ht]
\centering
 \begin{tikzpicture}



%
  
  \foreach \m in {0,0.5,1,1.5,2,2.5,3,3.5,4,4.5,5,5.5,6,6.5,7,7.5}
      \draw [dotted](0.5,\m)  -- (4.5,\m) ;
   \foreach \n in {0.5,1,1.5,2,2.5,3,3.5,4,4.5}
      \draw [dotted](\n,0)  -- (\n,8) ;

  \draw  (1,0)  -- (1,0.5) ;
  \draw  (1,7.5)  -- (1,8) ;
  \draw  (1.5,0)  -- (1.5,1.5) ;
  \draw  (1.5,6.5)  -- (1.5,8) ;
  \draw  (2,0)  -- (2,2.5) ;
  \draw  (2,5.5)  -- (2,8) ;
  \draw  (2.5,0)  -- (2.5,3.5) ;
  \draw  (2.5,4.5)  -- (2.5,8) ;
  \draw  (3,0)  -- (3,3) ;
  \draw  (3,5)  -- (3,8) ;
  \draw  (3.5,0)  -- (3.5,2) ;
  \draw  (3.5,6)  -- (3.5,8) ;
  \draw  (4,0)  -- (4,1) ;
  \draw  (4,7)  -- (4,8) ;
  
  \draw  (1,1.5)  -- (1.5,1.5) ;
  \draw  (3.5,2)  -- (4,2) ;
  \draw  (1,2.5)  -- (2,2.5) ;
  \draw  (3,3)  -- (4,3) ;
  \draw  (1,3.5)  -- (2.5,3.5) ;
  \draw  (1,4)  -- (4,4) ;
  \draw  (1,4.5)  -- (2.5,4.5) ;
  \draw  (3,5)  -- (4,5) ;
  \draw  (1,5.5)  -- (2,5.5) ;
  \draw  (3.5,6)  -- (4,6) ;
  \draw  (1,6.5)  -- (1.5,6.5) ;

 \draw (2.5,-0.5)  node [below] {$\omega_0$ if $X=\bar{x}$};



 \foreach \m in {0,0.5,1,1.5,2,2.5,3,3.5,4,4.5,5,5.5,6,6.5,7,7.5}
      \draw [dotted](6.5,\m)  -- (10.5,\m) ;
   \foreach \n in {0.5,1,1.5,2,2.5,3,3.5,4,4.5}
      \draw [dotted](\n+6,0)  -- (\n+6,8) ;
  
  \draw  (6.5,0)  -- (6.5,8) ;
  \draw  (10.5,0)  -- (10.5,8) ;
  \foreach \n in {0,0.5,1,1.5,2,2.5,3,3.5,4,4.5,5,5.5,6,6.5,7,7.5}
	\foreach \m in {6.5,10}
		\draw  (\m,\n)-- ++(0.5,0);
  
  \draw  (7,0)  -- (7,0.5) ;
  \draw  (7,7.5)  -- (7,8) ;
  \draw  (7.5,0)  -- (7.5,1.5) ;
  \draw  (7.5,6.5)  -- (7.5,8) ;
  \draw  (8,0)  -- (8,2.5) ;
  \draw  (8,5.5)  -- (8,8) ;
  \draw  (8.5,0)  -- (8.5,3.5) ;
  \draw  (8.5,4.5)  -- (8.5,8) ;
  \draw  (9,0)  -- (9,3) ;
  \draw  (9,5)  -- (9,8) ;
  \draw  (9.5,0)  -- (9.5,2) ;
  \draw  (9.5,6)  -- (9.5,8) ;
  \draw  (10,0)  -- (10,1) ;
  \draw  (10,7)  -- (10,8) ;
  
  \draw  (7,1.5)  -- (7.5,1.5) ;
  \draw (9.5,2)  -- (10,2) ;
  \draw  (7,2.5)  -- (8,2.5) ;
  \draw  (9,3)  -- (10,3) ;
  \draw  (7,3.5)  -- (8.5,3.5) ;
  \draw  (7,4)  -- (10,4) ;
  \draw  (7,4.5)  -- (8.5,4.5) ;
  \draw  (9,5)  -- (10,5) ;
  \draw  (7,5.5)  -- (8,5.5) ;
  \draw  (9.5,6)  -- (10,6) ;
  \draw  (7,6.5)  -- (7.5,6.5) ;
  
  \draw (8.5,-0.5)  node [below] {$\omega'_0$ if $Y=\bar{x}$};

\end{tikzpicture}
 \caption{$\omega_0$ and $\omega'_0$ on the fixed configurations for $G=\mathbb{Z}$, $U=\{-3, -2, -1, 0, 1, 2, 3\}$.} \label{fig:fixedconfigmulti}
\end{figure}

 To check that the last property stated in \eqref{eq:prop_coup2} is satisfied, let us inspect $C_0(A,\omega_0)$ and~$C_0(A,\omega'_0)$ in all possible cases listed inside the probability in \eqref{eq:speccoupling}: 
 \begin{itemize}
  \setlength\itemsep{0.3em}
  \item if~$X=Y$, then~$\omega_0(e)\leq \omega_0'(e)$ for every~$e\in \mathbb{E}_0$, thus~$C_0(A,\omega_0)\subseteq C_0(A,\omega'_0)$ for all~$A$;
  \item if~$X=\bar{x}$, then~$C_0(A,\omega_0)=A \subseteq C_0(A, \omega'_0)$ for all~$A$;
  \item if~$Y=\bar{x}$, then~$C_0(A,\omega'_0)=\partial \mathbb{V}_0 \supseteq C_0(A,\omega_0)$ for all~$A$.
 \end{itemize}
 Hence in all cases~$C_0(A,\omega_0)\subseteq C_0(A,\omega'_0)$ for every~$A\subseteq \partial \mathbb{V}_0$. We then let~$\mu_0$ be the distribution of~$(\omega_0,\omega'_0)$, completing the proof.
  
\end{proof}

\section{Proof of Theorem \ref{thm:orpercmulti}}
\label{s:proof2}
We start with a similar reduction to a particular case as the one in the beginning of the previous section.
As the proof of Claim~\ref{claim:contmm-1} did not rely on any special properties of~$\mathbb{G}$ (that~$\vec{\mathbb{G}}$ does not have), we can repeat the same argument in the oriented case.
We fix~$r \in \mathbb{N}$,~$u_0 \in V$ and let~$U:=B_r(u_0)$ as in the unoriented case. \emph{From now on, we assume that the edges~$e_1,\ldots, e_K$ of \eqref{eq:edges2} are all the edges with both endpoints belonging to~$U$.}

We again obtain the desired statement of Theorem~\ref{thm:orpercmulti} as a consequence of the following claim.
\begin{claim}\label{claim:orpqsurv}
 For all~$p\in(0,1)$,~$\bold{q_0}\in(0,1)^{K}$ and~$\epsilon\in(0, 1-p)$ there exists a~$\delta>0$ such that for any~$\bold{q}, \bold{q'}\in(0,1)^{K}$ satisfying~$\|\bold{q_0}-\bold{q}\|_{\infty}<\delta$ and~$\|\bold{q_0}-\bold{q'}\|_{\infty}<\delta$ we have
 \[
 \vec{\mathbb{P}}_{\bold{q}, p}(\vec{C}_{\infty})\leq\vec{\mathbb{P}}_{\bold{q'}, p+\epsilon}(\vec{C}_{\infty}).
 \]
\end{claim}

 Theorem~\ref{thm:orpercmulti} follows from this claim by the same argument as in the unoriented case, so we omit the details.

 \begin{remark}\label{rem:nonor}
 The proof of Claim~\ref{claim:orpqsurv} is similar to that of Claim~\ref{claim:pqsurv} but slightly more involved.
In the proof of the unoriented case we used Lemma~\ref{lemma:coupling1} with a single determinisitic configuration~$\bar{x}=(\bar{x}^U, \bar{x}^{\partial, 1}, \bar{x}^{\partial, 2})$.
This was possible because our choice of~$\bar{x}$ was such that, for every~$\omega_0\in\Omega_0$ and~$A\subseteq\partial\mathbb{V}_0$ we have
\begin{align*}
 &C_0(A, (\bar{x}^U, \bar{x}^{\partial, 1}))=A \subseteq C_0(A, \omega_0),\\
 &C_0(A, (\bar{x}^U, \bar{x}^{\partial, 1} \vee \bar{x}^{\partial, 2}))=\partial\mathbb{V}_0 \supseteq C_0(A, \omega_0).
\end{align*}
However, we cannot find a configuration with similar properties in the oriented case (see Remark~\ref{rem:orproblem} at the end of the proof).
\end{remark}

\begin{proof}[Proof of Claim~\ref{claim:orpqsurv}]
Let
\[
 \mathbb{V}_n=\{(v, m)\in\mathbb{V}: v\in B_{r+1}(u_0), (2K+2)n\leq m \leq (2K+2)(n+1)\}
\]
and 
\[
 \vec{\mathbb{E}}_n=\{ e\in\vec{\mathbb{E}}: e \text{ has both endpoints in }\mathbb{V}_n\}.
\]
Note that~$\vec{\mathbb{E}}_n$ are disjoint.
Next, recall the definition of~$\vec{\mathbb{E}}^{i}$ from \eqref{eq:eioriented} and define, for~$n\in\mathbb{Z}$ and~$1\leq i \leq K$,
$$ \vec{\mathbb{E}}^i_n = \vec{\mathbb{E}}_n \cap \vec{\mathbb{E}}^i,\qquad \vec{\mathbb{E}}^\partial_n = \vec{\mathbb{E}}_n \backslash \left(\cup_{i=1}^{K} \vec{\mathbb{E}}^i_n\right),\qquad \vec{\mathbb{E}}_\mathcal{O} = \vec{\mathbb{E}} \backslash \left( \cup_{n \in \mathbb{Z}} \vec{\mathbb{E}}_n\right).$$
The ``edge boundary''~$\vec{\mathbb{E}}^\partial_n$ consists of edges of the form~$\langle(u,m),(v,m+1)\rangle$, with~$u, v \in\mathbb{V}_n$ and at least one of~$u$ and~$v$ at distance~$r+1$ from~$u_0$. Define corresponding sets of configurations~$\vec{\Omega}_n^i$,~$\vec{\Omega}_n^\partial$ and~$\vec{\Omega}_\mathcal{O}$.

For each~$n$, define the boundary sets
\begin{align*}&\underline{\partial} \mathbb{V}_n = \{(v,n) \in \mathbb{V}_n: \text{dist}(v, u_0)=r+1\} \cup (\mathbb{V}_n \cap (\mathbb{V} \times \{(2K+2)n\})),\\
&\overline{\partial}\mathbb{V}_n =  \{(v,n) \in \mathbb{V}_n: \text{dist}(v, u_0)=r+1\} \cup (\mathbb{V}_n \cap (\mathbb{V} \times \{(2K+2)(n+1)\})),
\end{align*}
so that~$\underline{\partial} \mathbb{V}_n$ consists of ``walls and floor'' and~$\overline{\partial}\mathbb{V}_n$ consists of ``walls and ceiling'' of the box~$\mathbb{V}_n$. Given any~$\varnothing\neq A\subseteq \underline{\partial}\mathbb{V}_n$ and~$\omega_n \in \vec{\Omega}_n$, define
 \begin{align*}
  \vec{C}_n(A,\omega_n) = \{(v,n) \in  \overline{\partial} \mathbb{V}_n: (v_0,n_0) \stackrel{\omega_n}{\longrightarrow} (v,n) \text{ for some }(v_0,n_0) \in A\},
  \end{align*}
where the notation~$(v_0,n_0)  \stackrel{\omega_n}{\longrightarrow} (v,n)$ means that~$(v_0,n_0)$ and~$(v,n)$ are connected by an~$\omega_n$-open path of edges of~$\vec{\mathbb{E}}_n$.

Fix~$p$,~$\bold{q_0}$ and~$\epsilon$, and for~$\delta$ close enough to zero let~$\bold{q}=(q_1, \dots, q_{K})$ and~$\bold{q'}=(q'_1, \dots, q'_{K})$ be as in the statement of the claim. 
We will define coupling measures~$\vec{\mu}_\mathcal{O}$ on~$(\vec{\Omega}_\mathcal{O})^2$ and~$\vec{\mu}_n$ on~$(\vec{\Omega}_n)^2$ that satisfy similar properties as in the unoriented case. First, 
\begin{equation}\label{eq:orprop_coup1}\begin{split}&(\omega_\mathcal{O},\omega'_\mathcal{O}) \sim \vec{\mu}_\mathcal{O} \quad \Longrightarrow \quad \omega_\mathcal{O} \stackrel{\text{(d)}}{=} \vec{\mathbb{P}}_{\bold{q}, p}\vert_{\vec{\mathbb{E}}_\mathcal{O}},\quad \omega'_\mathcal{O} \stackrel{\text{(d)}}{=} \vec{\mathbb{P}}_{\bold{q'}, p+\epsilon}\vert_{\vec{\mathbb{E}}_\mathcal{O}} \\[.2cm]&\hspace{7cm}\text{ and } \omega_\mathcal{O} \leq \omega'_\mathcal{O} \text{ a.s.}\end{split}\end{equation}
Second, 
\begin{equation}\label{eq:orprop_coup2}\begin{split}&(\omega_n,\omega'_n) \sim \vec{\mu}_n \quad \Longrightarrow \quad \omega_n \stackrel{\text{(d)}}{=} \vec{\mathbb{P}}_{\bold{q}, p}\vert_{\vec{\mathbb{E}}_n},\quad \omega'_n \stackrel{\text{(d)}}{=} \vec{\mathbb{P}}_{\bold{q'}, p+\epsilon}\vert_{\vec{\mathbb{E}}_n} \\[.2cm]&\hspace{4cm}\text{ and } \vec{C}_n(A,\omega_n) \subseteq \vec{C}_n(A,\omega'_n) \text{ for all } A \subset \underline{\partial} \mathbb{V}_n \text{ a.s.}\end{split}\end{equation}
We then define the coupling measure~$\vec{\mu}$ on~$\vec{\Omega}^2$ by
$$\vec{\mu} = \vec{\mu}_\mathcal{O} \otimes \left(\otimes_{n \in \mathbb{Z}} \vec{\mu}_n \right).$$
It is clear from \eqref{eq:orprop_coup1} and \eqref{eq:orprop_coup2} that, if~$(\omega, \omega') \sim \vec{\mu}$, then~$\omega \sim \vec{\mathbb{P}}_{\bold{q}, p}$,~$\omega' \sim \vec{\mathbb{P}}_{\bold{q'}, p+\epsilon}$, and almost surely if~$\vec{C}_{\infty}$ holds for~$\omega$, then it holds for~$\omega'$.
Consequently
\[
 \vec{\mathbb{P}}_{\bold{q}, p}(\vec{C}_{\infty})\leq\vec{\mathbb{P}}_{\bold{q'}, p+\epsilon}(\vec{C}_{\infty}).
\]

The measure~$\vec{\mu}_\mathcal{O}$ is defined using the same standard coupling as the corresponding measure in the proof of Claim~\ref{claim:pqsurv}. The measures~$\vec{\mu}_n$ will again be taken as translations of each other, so we only define~$\vec{\mu}_0$.
The construction relies on Lemma~\ref{lemma:coupling2}. The finite set~$S$ and the decomposition~$S = \hat{S} \cup \doublehat{S}$ of the statement of that lemma are given by
\begin{align*}S= \vec{\Omega}_0^1 \times \cdots \times \vec{\Omega}^{K}_0 \times \vec{\Omega}^\partial_0 \times \vec{\Omega}^\partial_0,\hspace{8pt} \hat{S} = \vec{\Lambda}_0^1 \times \cdots \times \vec{\Lambda}^K_0 \times \vec{\Omega}^\partial_0 \times \vec{\Omega}^\partial_0,\hspace{8pt} \doublehat{S} = S \backslash \hat{S},
\end{align*}
where~$\vec{\Lambda}_0^i$ is the set of configurations in~$\vec{\Omega}_0^i$ in which edges from height~$K$ to height~$K+1$ are closed. The definition of~$\hat{x}$ and~$\doublehat{x}$ is as follows (see Figure~\ref{fig:xstarbar} for a specific example):
 
 \begin{figure}[ht]
\centering
 \begin{tikzpicture}
 \foreach \n in {1,4.5,8.5}
      \foreach \m in {-0.1,-6.1}
	    \draw (\n,\m)  -- ++(0,0.2) ;

   \draw (0,0)  -- (2.5,0) ;
 \draw (1,-0.1)  node [below] {0};

  \draw [dashed] (0.5,0)  -- (0.5,4) ;
  \draw [dashed] (2,0)  -- (2,4) ;
 \draw [dashed] (0.5,4)  -- (2,4) ;
 \draw (1.25,-0.5)  node [below] {$\hat{x}^U$};

\foreach \n in {1,2}
	\foreach \m in {0,1, 2, 3}
		\draw [->] (\n,\m)-- ++(-0.5,0.5);
\foreach \n in {1,2}
	\foreach \m in {1,3,4}
		\draw [<-] (\n,\m)-- ++(-0.5,-0.5);
\foreach \m in {0,1,2, 3}
		\draw [->] (1,\m)-- ++(0.5,0.5);
\foreach \m in {1, 3,4}
		\draw [<-] (1,\m)-- ++(0.5,-0.5);

 \draw (3,0)  -- (6.5,0) ;
 \draw (4.5,-0.1)  node [below] {0};
 \draw [dashed] (3.5,0)  -- (3.5,4) ;
 \draw [dashed] (4,0)  -- (4,4) ;
 \draw [dashed] (6,0)  -- (6,4) ;
 \draw [dashed] (5.5,0)  -- (5.5,4) ;
 \draw [dashed] (3.5,4)  -- (4,4) ;
 \draw [dashed] (5.5,4)  -- (6,4) ;
\draw (4.75,-0.5)  node [below] {$\hat{x}^{\partial, 1}$};

 \draw (7,0)  -- (10.5,0) ;
 \draw (8.5,-0.1)  node [below] {0};
 \draw [dashed] (7.5,0)  -- (7.5,4) ;
 \draw [dashed] (8,0)  -- (8,4) ;
 \draw [dashed] (10,0)  -- (10,4) ;
 \draw [dashed] (9.5,0)  -- (9.5,4) ;
 \draw [dashed] (7.5,4)  -- (8,4) ;
 \draw [dashed] (9.5,4)  -- (10,4) ;
\draw (8.75,-0.5)  node [below] {$\hat{x}^{\partial, 2}$};

\foreach \n in {7.5,9.5}
	\foreach \m in {0,1,2,3}
		\draw [->] (\n,\m)-- ++(0.5,0.5);
\foreach \n in {7.5,9.5}
	\foreach \m in {1,2,3,4}
		\draw [<-] (\n,\m)-- ++(0.5,-0.5);

\draw (0,-6)  -- (2.5,-6) ;
 \draw (1,-6.1)  node [below] {0};
  
  \draw [dashed] (0.5,-6)  -- (0.5,-2) ;
  \draw [dashed] (2,-6)  -- (2,-2) ;
 \draw [dashed] (0.5,-2)  -- (2,-2) ;
 \draw (1.25,-6.5)  node [below] {$\doublehat{x}^U$};

\foreach \n in {1,2}
	\foreach \m in {0,1, 2, 3}
		\draw [->] (\n,\m-6)-- ++(-0.5,0.5);
\foreach \n in {1,2}
	\foreach \m in {1,2,3,4}
		\draw [<-] (\n,\m-6)-- ++(-0.5,-0.5);
\foreach \m in {0,1,2, 3}
		\draw [->] (1,\m-6)-- ++(0.5,0.5);
\foreach \m in {1,2, 3,4}
		\draw [<-] (1,\m-6)-- ++(0.5,-0.5);

 \draw (3,-6)  -- (6.5,-6) ;
 \draw (4.5,-6.1)  node [below] {0};
 \draw [dashed] (3.5,-6)  -- (3.5,-2) ;
 \draw [dashed] (4,-6)  -- (4,-2) ;
 \draw [dashed] (6,-6)  -- (6,-2) ;
 \draw [dashed] (5.5,-6)  -- (5.5,-2) ;
 \draw [dashed] (3.5,-2)  -- (4,-2) ;
 \draw [dashed] (5.5,-2)  -- (6,-2) ;
\draw (4.75,-6.5)  node [below] {$\doublehat{x}^{\partial, 1}$};

 \draw (7,-6)  -- (10.5,-6) ;
 \draw (8.5,-6.1)  node [below] {0};
 \draw [dashed] (7.5,-6)  -- (7.5,-2) ;
 \draw [dashed] (8,-6)  -- (8,-2) ;
 \draw [dashed] (10,-6)  -- (10,-2) ;
 \draw [dashed] (9.5,-6)  -- (9.5,-2) ;
 \draw [dashed] (7.5,-2)  -- (8,-2) ;
 \draw [dashed] (9.5,-2)  -- (10,-2) ;
\draw (8.75,-6.5)  node [below] {$\doublehat{x}^{\partial, 2}$};

\foreach \n in {7.5,9.5}
	\foreach \m in {0,1,2,3}
		\draw [->] (\n,\m-6)-- ++(0.5,0.5);
\foreach \n in {7.5,9.5}
	\foreach \m in {1,2,3,4}
		\draw [<-] (\n,\m-6)-- ++(0.5,-0.5);

\end{tikzpicture}
 \caption{The deterministic configurations for $G=\mathbb{Z}$ and $U=\{-1, 0, 1, 2\}$. In this case~${K=3}$. Note that only one of the two disjoint subgraphs of~$\vec{\mathbb{G}}$ is displayed.} \label{fig:xstarbar}
\end{figure}
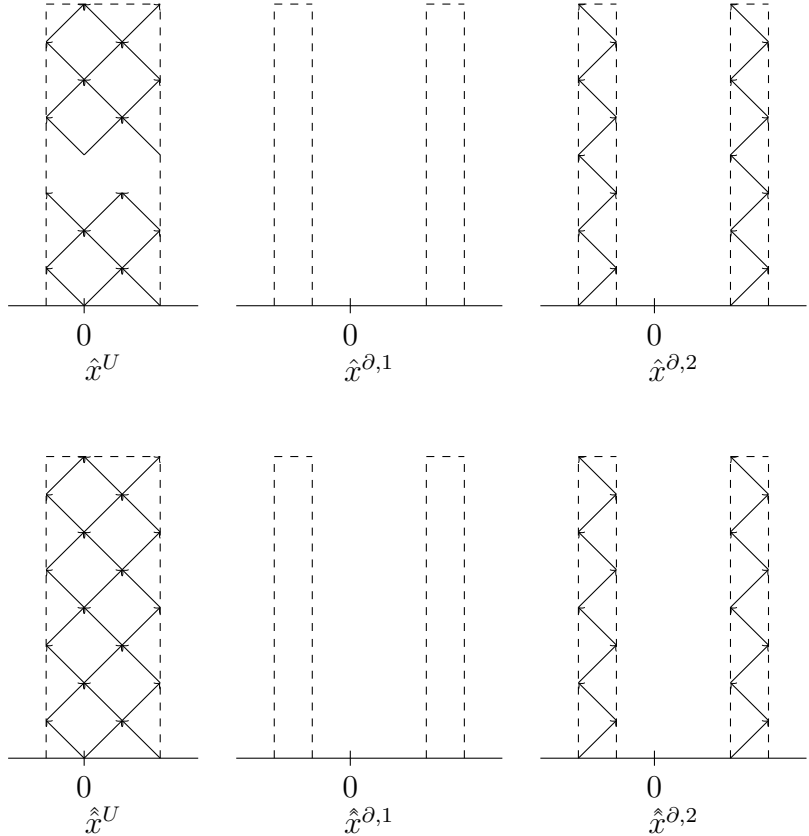
 
 \begin{itemize}
 \setlength\itemsep{0.3em}
 \item~$\hat{x}= (\hat{x}^1,\ldots, \hat{x}^K, \hat{x}^{\partial, 1}, \hat{x}^{\partial, 2})$ with~$\hat{x}^i \in \vec{\Lambda}_{0}^i$ and~$\hat{x}^{\partial, 1}, \hat{x}^{\partial, 2} \in \vec{\Omega}_0^{\partial}$;
  \item~$\doublehat{x}= (\doublehat{x}^1,\ldots, \doublehat{x}^K, \doublehat{x}^{\partial, 1}, \doublehat{x}^{\partial, 2})$ with~$\doublehat{x}^{i} \in \vec{\Omega}_{0}^i\setminus\vec{\Lambda}_{0}^i$ and~$\doublehat{x}^{\partial, 1}, \doublehat{x}^{\partial, 2} \in \vec{\Omega}_0^{\partial}$;
  \item~$\hat{x}^{\partial, 1}\equiv 0$,~$\hat{x}^{\partial, 2}\equiv 1$ and for each~$i$, ~$\hat{x}^{i}(e)=0$  if and only if~$e$ goes from height~$K$ to~$K+1$,;
  \item~$\doublehat{x}^{\partial, 1}\equiv 0$, ~$\doublehat{x}^{\partial, 2}\equiv 1$ and for each~$i$,~$\doublehat{x}^{i}\equiv 1$.
 \end{itemize}

 By Lemma~\ref{lemma:coupling2}, if~$\delta$ is close enough to zero,  there exists a coupling of~$(K+2)$-tuples of configurations
\begin{align*}
  X=(X^1, \dots, X^{K}, X^{\partial, 1}, X^{\partial, 2}), \;Y&=(Y^1, \dots, Y^{K}, Y^{\partial, 1}, Y^{\partial, 2}) \\
  &\in \vec{\Omega}^1_0 \times \dots \times \vec{\Omega}^K_0 \times \vec{\Omega}_0^{\partial} \times \vec{\Omega}_0^{\partial}
  \end{align*}
  such that
 \begin{itemize}
  \setlength\itemsep{0.3em}
  \item the values of~$X^1, \dots, X^K, X^{\partial, 1}, X^{\partial, 2}$ are independent on all edges;
  \item the values of~$Y^1, \dots, Y^K, Y^{\partial, 1}, Y^{\partial, 2}$ are independent on all edges;
  \item~$X^i$ assigns each edge to be open with probability~$q_i$;
  \item~$Y^i$ assigns each edge to be open with probability~$q'_i$;
  \item~$X^{\partial, 1}$ and~$Y^{\partial, 1}$ assign each edge to be open with probability~$p$;
  \item~$X^{\partial, 2}$ and~$Y^{\partial, 2}$ assign each edge to be open with probability~$\frac{\epsilon}{1-p}$;
  \item~$(X, Y)$ satisfies \begin{equation}\label{eq:orspeccoupling}
 \P\left(\{X=Y\}\cup\{X=\hat{x}\}\cup\{X\in \hat{S}\cup\{\doublehat{x}\}, Y=\hat{x}\} \cup\{Y=\doublehat{x}\} \right)=1.
\end{equation}
 \end{itemize}
 
  Now let~$\omega_0=(X^1, \dots, X^K, X^{\partial, 1})$ and~$\omega_0'=(Y^1, \dots, Y^K, Y^{\partial, 1}\vee Y^{\partial, 2})$. 
 Thus~$\omega_0'$ assigns each edge in~$\vec{\mathbb{E}}_0^{\partial}$ to be open with probability~$p+\epsilon$. 
 See Figure~\ref{fig:orfixedconfigsmulti} for~$\omega_0$ and~$\omega'_0$ if~$X$ or~$Y$ equals~$\hat{x}$ or~$\doublehat{x}$.
 
 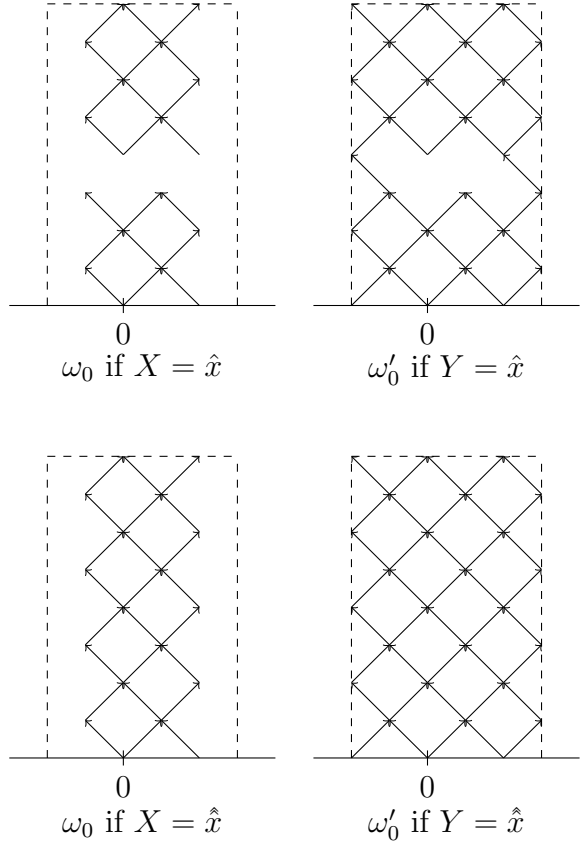
\begin{figure}[ht]
\centering
 \begin{tikzpicture}
  \foreach \n in {4.5,8.5}
      \foreach \m in {-0.1,-6.1}
	    \draw (\n,\m)  -- ++(0,0.2) ;

 \draw (3,0)  -- (6.5,0) ;
 \draw (4.5,-0.1)  node [below] {0};
 \draw [dashed] (3.5,0)  -- (3.5,4) ;
  \draw [dashed] (6,0)  -- (6,4) ;
 \draw [dashed] (3.5,4)  -- (6,4) ;
\draw (4.75,-0.5)  node [below] {$\omega_0$ if $X=\hat{x}$};

\foreach \n in {1,2}
	\foreach \m in {0,1, 2, 3}
		\draw [->] (\n+3.5,\m)-- ++(-0.5,0.5);
\foreach \n in {1,2}
	\foreach \m in {1,3,4}
		\draw [<-] (\n+3.5,\m)-- ++(-0.5,-0.5);
\foreach \m in {0,1,2, 3}
		\draw [->] (4.5,\m)-- ++(0.5,0.5);
\foreach \m in {1, 3,4}
		\draw [<-] (4.5,\m)-- ++(0.5,-0.5);

 \draw (7,0)  -- (10.5,0) ;
 \draw (8.5,-0.1)  node [below] {0};
 \draw [dashed] (7.5,0)  -- (7.5,4) ;
 \draw [dashed] (10,0)  -- (10,4) ;
 \draw [dashed] (7.5,4)  -- (10,4) ;
\draw (8.75,-0.5)  node [below] {$\omega'_0$ if $Y=\hat{x}$};

\foreach \n in {7.5,9.5}
	\foreach \m in {0,1,2,3}
		\draw [->] (\n,\m)-- ++(0.5,0.5);
\foreach \n in {7.5,9.5}
	\foreach \m in {1,2,3,4}
		\draw [<-] (\n,\m)-- ++(0.5,-0.5);
\foreach \n in {1,2}
	\foreach \m in {0,1, 2, 3}
		\draw [->] (\n+7.5,\m)-- ++(-0.5,0.5);
\foreach \n in {1,2}
	\foreach \m in {1,3,4}
		\draw [<-] (\n+7.5,\m)-- ++(-0.5,-0.5);
\foreach \m in {0,1,2, 3}
		\draw [->] (8.5,\m)-- ++(0.5,0.5);
\foreach \m in {1, 3,4}
		\draw [<-] (8.5,\m)-- ++(0.5,-0.5);

 \draw (3,-6)  -- (6.5,-6) ;
 \draw (4.5,-6.1)  node [below] {0};
 \draw [dashed] (3.5,-6)  -- (3.5,-2) ;
 \draw [dashed] (6,-6)  -- (6,-2) ;
 \draw [dashed] (3.5,-2)  -- (6,-2) ;
\draw (4.75,-6.5)  node [below] {$\omega_0$ if $X=\doublehat{x}$};

\foreach \n in {1,2}
	\foreach \m in {0,1, 2, 3}
		\draw [->] (\n+3.5,\m-6)-- ++(-0.5,0.5);
\foreach \n in {1,2}
	\foreach \m in {1,2,3,4}
		\draw [<-] (\n+3.5,\m-6)-- ++(-0.5,-0.5);
\foreach \m in {0,1,2, 3}
		\draw [->] (4.5,\m-6)-- ++(0.5,0.5);
\foreach \m in {1,2, 3,4}
		\draw [<-] (4.5,\m-6)-- ++(0.5,-0.5);

 \draw (7,-6)  -- (10.5,-6) ;
 \draw (8.5,-6.1)  node [below] {0};
 \draw [dashed] (7.5,-6)  -- (7.5,-2) ;
 \draw [dashed] (10,-6)  -- (10,-2) ;
 \draw [dashed] (7.5,-2)  -- (10,-2) ;
\draw (8.75,-6.5)  node [below] {$\omega'_0$ if $Y=\doublehat{x}$};

\foreach \n in {7.5,9.5}
	\foreach \m in {0,1,2,3}
		\draw [->] (\n,\m-6)-- ++(0.5,0.5);
\foreach \n in {7.5,9.5}
	\foreach \m in {1,2,3,4}
		\draw [<-] (\n,\m-6)-- ++(0.5,-0.5);		
   \foreach \n in {1,2}
	\foreach \m in {0,1, 2, 3}
		\draw [->] (\n+7.5,\m-6)-- ++(-0.5,0.5);
\foreach \n in {1,2}
	\foreach \m in {1,2,3,4}
		\draw [<-] (\n+7.5,\m-6)-- ++(-0.5,-0.5);
\foreach \m in {0,1,2, 3}
		\draw [->] (8.5,\m-6)-- ++(0.5,0.5);
\foreach \m in {1,2, 3,4}
		\draw [<-] (8.5,\m-6)-- ++(0.5,-0.5);

\end{tikzpicture}
 \caption{$\omega_0$ and $\omega'_0$ on the fixed configurations for $G=\mathbb{Z}$, $U=\{-1, 0, 1, 2\}$.} \label{fig:orfixedconfigsmulti}
\end{figure}
 
 To check that the last property in \eqref{eq:orprop_coup2} is satisfied, we need to show that in any of the situations listed inside the probability in~\eqref{eq:orspeccoupling}, we have~$\vec{C}_0(A, \omega_0) \subseteq \vec{C}_0(A, \omega'_0)$ for any~$\varnothing\neq A \subseteq \underline{\partial}\mathbb{V}_n$.~$\{X = \hat{x}\}$ entails~$\vec{C}_0(A, \omega_0)=A\cap\overline{\partial}\mathbb{V}_0$ and~$\{X = Y\}$,~$\{X \in \hat{S},\; Y = \hat{x}\}$ as well as~$\{Y = \doublehat{x}\}$ lead to~$\omega_0(e)\leq\omega_0'(e)$ for every~$e\in\vec{\mathbb{E}}_0$. 
 The remaining case is when~$X = \doublehat{x}$ and~$Y = \hat{x}$.  In this case,~$(v_0, n_0)\xrightarrow[]{\omega_0}(v_1, n_1)$ can only happen if~$v_0, v_1 \in U, n_0=0$ and~$n_1=(2K+2)$. But then we also have~$(v_0, n_0)\xrightarrow[]{\omega'_0}(v_1, n_1)$.

Finally, we let~$\vec{\mu}_0$ be the distribution of~$(\omega_0, \omega'_0)$, completing the proof.

\end{proof}

\begin{remark}\label{rem:orproblem}
In the oriented case we cannot find a configuration with similar properties as the one in Remark~\ref{rem:nonor}.
If~$\hat{x}=(\hat{x}^U, \hat{x}^{\partial, 1}, \hat{x}^{\partial, 2})$ is such that~$\hat{x}^U$ contains at least one closed edge, 
depending on the topolgy of~$G\vert_{U}$, the induced subgraph of~$G$ on~$U$, we can find a configuration~$\omega_0\in\vec{\Omega}_0$ and a set~$A\subseteq \underline{\partial} \mathbb{V}_0$ such that
\[
\vec{C}_0(A, (\hat{x}^U, \hat{x}^{\partial, 1} \vee \hat{x}^{\partial, 2})) \nsupseteq \vec{C}_0(A, \omega_0).
\]
In case~$\doublehat{x}=(\doublehat{x}^U, \doublehat{x}^{\partial, 1}, \doublehat{x}^{\partial, 2})$ is such that every edge in~$\doublehat{x}^U$ is open, then we can always find a configuration~$\omega'_0\in\vec{\Omega}_0$ and a set~$B\subseteq \underline{\partial} \mathbb{V}_0$ such that
\[
\vec{C}_0(B, (\doublehat{x}^U, \doublehat{x}^{\partial, 1})) \nsubseteq \vec{C}_0(B, \omega'_0).
\]
(See Figure~\ref{fig:counterex} for  examples).

 \begin{figure}[ht]
\centering
 \begin{tikzpicture}
  \foreach \n in {4.5,8.5}
      \foreach \m in {-0.1,-6.1}
	    \draw (\n,\m)  -- ++(0,0.2) ;

 \draw (3,0)  -- (6.5,0) ;
 \draw (4.5,-0.1)  node [below] {0};
 \draw [dashed] (3.5,0)  -- (3.5,4) ;
  \draw [dashed] (6,0)  -- (6,4) ;
 \draw [dashed] (3.5,4)  -- (6,4) ;
\draw (4.75,-0.5)  node [below] {$\omega_0$};

\foreach \n in {0,0.5,1,1.5, 2}
		\draw [->] (3.5+\n,1+\n)-- ++(0.5,0.5);

 \draw (7,0)  -- (10.5,0) ;
 \draw (8.5,-0.1)  node [below] {0};
 \draw [dashed] (7.5,0)  -- (7.5,4) ;
 \draw [dashed] (10,0)  -- (10,4) ;
 \draw [dashed] (7.5,4)  -- (10,4) ;
\draw (8.75,-0.5)  node [below] {$(\hat{x}^U, \hat{x}^{\partial, 1} \vee \hat{x}^{\partial, 2})$};

\foreach \Point in {(6,3.5), (9.5,4), (8.5,4), (7.5,4), (7.5,3), (7.5,2), (5.5,-2), (4.5,-2), (9.5,-2)}{
    \node at \Point {\textbullet};
}

\foreach \Point in {(3.5,1), (7.5,1), (5.5,-6), (9.5,-6)}{
    \node at \Point {$\circ$};
}

\foreach \n in {7.5,9.5}
	\foreach \m in {0,1,2,3}
		\draw [->] (\n,\m)-- ++(0.5,0.5);
\foreach \n in {7.5,9.5}
	\foreach \m in {1,2,3,4}
		\draw [<-] (\n,\m)-- ++(0.5,-0.5);
\foreach \n in {1,2}
	\foreach \m in {0,1, 2, 3}
		\draw [->] (\n+7.5,\m)-- ++(-0.5,0.5);
\foreach \n in {1,2}
	\foreach \m in {1,3,4}
		\draw [<-] (\n+7.5,\m)-- ++(-0.5,-0.5);
\foreach \m in {0,1,2, 3}
		\draw [->] (8.5,\m)-- ++(0.5,0.5);
\foreach \m in {1, 3,4}
		\draw [<-] (8.5,\m)-- ++(0.5,-0.5);

 \draw (3,-6)  -- (6.5,-6) ;
 \draw (4.5,-6.1)  node [below] {0};
 \draw [dashed] (3.5,-6)  -- (3.5,-2) ;
 \draw [dashed] (6,-6)  -- (6,-2) ;
 \draw [dashed] (3.5,-2)  -- (6,-2) ;
\draw (4.75,-6.5)  node [below] {$(\doublehat{x}^U, \doublehat{x}^{\partial, 1})$};

\foreach \n in {1,2}
	\foreach \m in {0,1, 2, 3}
		\draw [->] (\n+3.5,\m-6)-- ++(-0.5,0.5);
\foreach \n in {1,2}
	\foreach \m in {1,2,3,4}
		\draw [<-] (\n+3.5,\m-6)-- ++(-0.5,-0.5);
\foreach \m in {0,1,2, 3}
		\draw [->] (4.5,\m-6)-- ++(0.5,0.5);
\foreach \m in {1,2, 3,4}
		\draw [<-] (4.5,\m-6)-- ++(0.5,-0.5);

 \draw (7,-6)  -- (10.5,-6) ;
 \draw (8.5,-6.1)  node [below] {0};
 \draw [dashed] (7.5,-6)  -- (7.5,-2) ;
 \draw [dashed] (10,-6)  -- (10,-2) ;
 \draw [dashed] (7.5,-2)  -- (10,-2) ;
\draw (8.75,-6.5)  node [below] {$\omega'_0$};

\foreach \m in {0,1, 2, 3}
		\draw [->] (9.5,\m-6)-- ++(-0.5,0.5);
\foreach \m in {1,2,3,4}
 		\draw [<-] (9.5,\m-6)-- ++(-0.5,-0.5);

\end{tikzpicture}
 \caption{Examples of why we need two configurations in the oriented case. \textbullet\ denotes the vertices of $C_0(\circ, \cdot) \backslash \{\circ\}$ in each configuration ($G=\mathbb{Z}$, $U=\{-1, 0, 1, 2\}$).} \label{fig:counterex}
\end{figure}
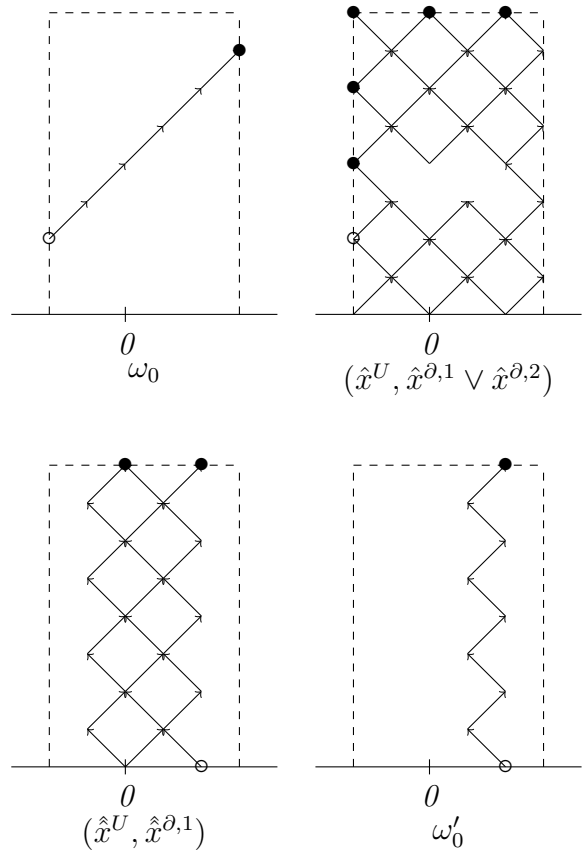

\pagebreak
 This is the reason why we needed to apply Lemma~\ref{lemma:coupling2}, involving two deterministic configurations, to make the coupling work. The trick was to choose~$\hat{x}$ and~$\doublehat{x}$ in a way that for every~$A\subseteq \underline{\partial} \mathbb{V}_0$,
\[
 \vec{C}_0(A, (\doublehat{x}^U, \doublehat{x}^{\partial, 1})) \subseteq \vec{C}_0(A, (\hat{x}^U, \hat{x}^{\partial, 1} \vee \hat{x}^{\partial, 2})).
\]

\end{remark}

 \begin{remark}\label{rem:contact_process}
As mentioned in Section~\ref{ss:contact}, the approach we used to prove Theorem~\ref{thm:orpercmulti} is not readily applicable when the oriented model is replaced by a ``continuous-time'' version such as the contact process. The essential difficulty is that our approach involves finding a configuration that is better than any other in connecting points of any possible boundary set~$A$ to other boundary points.
 In a continuous-time setting, the set of configurations inside a finite box is infinite, so such an optimal configuration cannot exist (in a standard construction involving Poisson processes, one can always introduce extra arrivals between those of a fixed configuration). As a potential strategy, one could attempt to sophisticate our method by partitioning the configuration space not in two, but in infinitely many parts, proving a corresponding version of Lemma~\ref{lemma:coupling2}, and finding a sequence of finer and finer configurations which could produce an effective coupling. 
\end{remark}

\vspace{.5cm}
\textbf{Acknowledgments.}
The authors would like to thank the anonymous referee for the thorough review of the paper and for the insightful suggestions and comments.

\end{document}

%% file: graphs.eps_tex
\begingroup%
  \makeatletter%
  \providecommand\color[2][]{%
    \errmessage{(Inkscape) Color is used for the text in Inkscape, but the package 'color.sty' is not loaded}%
    \renewcommand\color[2][]{}%
  }%
  \providecommand\transparent[1]{%
    \errmessage{(Inkscape) Transparency is used (non-zero) for the text in Inkscape, but the package 'transparent.sty' is not loaded}%
    \renewcommand\transparent[1]{}%
  }%
  \providecommand\rotatebox[2]{#2}%
  \ifx\svgwidth\undefined%
    \setlength{\unitlength}{257.69748586bp}%
    \ifx\svgscale\undefined%
      \relax%
    \else%
      \setlength{\unitlength}{\unitlength * \real{\svgscale}}%
    \fi%
  \else%
    \setlength{\unitlength}{\svgwidth}%
  \fi%
  \global\let\svgwidth\undefined%
  \global\let\svgscale\undefined%
  \makeatother%
  \begin{picture}(1,0.27)%
    \put(0,0){\includegraphics[width=\unitlength]{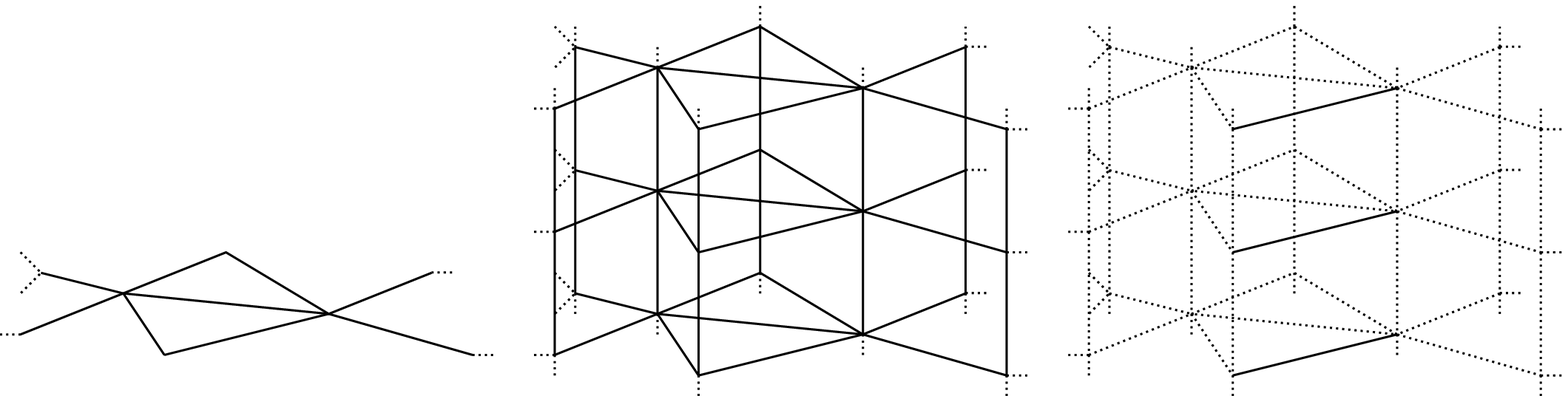}}%
    \put(0.095,-0.015){\color[rgb]{0,0,0}\makebox(0,0)[lb]{\smash{$G$}}}%
    \put(0.498,-0.015){\color[rgb]{0,0,0}\makebox(0,0)[lb]{\smash{$\mathbb{G}$}}}%
    \put(0.842, 0.0){\color[rgb]{0,0,0}\makebox(0,0)[lb]{\smash{$\mathbb{E}''$}}}%
  \end{picture}%
\endgroup%

%% file: percolation2.bbl
\begin{thebibliography}{1}

\bibitem{AG91}
Aizenman, M.\, Grimmett, G.\: Strict monotonicity for critical points in percolation and ferromagnetic models. 
J. Stat. Phys. 63, no.\ 5-6, 817--835 (1991)


\bibitem{BBR14}
Balister, P.\, Bollob\'as, B.\, Riordan, O.\: Essential enhancements revisited.
arXiv:1402.0834 (2014)

\bibitem{BR06}
Belitsky, V.\, Ritchie, T.\ L.\: Improved lower bounds for the critical probability of oriented bond percolation in two dimensions.
J. Stat. Phys. 122, no.\ 2, 279--302 (2006)

\bibitem{BS96}
Benjamini, I.\, Schramm, O.\: Percolation beyond~$\mathbb{Z}^d$, many questions and a few answers.
Electron. Comm. Probab. 1, no.\ 8, 71-–82 (1996)

\bibitem{CT17}
Candellero, E.\, Teixeira, A.\: Percolation and isoperimetry on roughly transitive graphs.
arXiv:1507.07765 (2017)

\bibitem{GN90}
Grimmett, G.\ R.\, Newman, C.\ M.\: Percolation in~$\infty +1$ dimensions.
In: Grimmett, G. R., Welsh, D. J. A. (eds.) Disorder in physical systems, pp. 219--240. Clarendon Press, Oxford (1990)

\bibitem{J05}
Jung, P.\: The critical value of the contact process with added and removed edges.
J. Theor. Probab. 18, no.\ 4, 949--955 (2005)

\bibitem{L13}
Liggett, T.\ M.\: Stochastic interacting systems: contact, voter and exclusion processes. Vol. 324, Springer science \& Business Media (2013)

\bibitem{LP16}
Lyons, R.\, Peres, Y.\: Probability on trees and networks. Vol. 42, Cambridge University Press (2016)

\bibitem{LRV17}
de Lima, B.\ N.\ B.\, Rolla, L.\ T.\, Valesin, D.\: Monotonicity and phase diagram for multi-range percolation on oriented trees.
arXiv:1702.03841 (2017)

\bibitem{SV16}
Szab\'o, R.\, Valesin, D.\: From survival to extinction of the contact process by the removal of a single edge.
Electron. Comm. Probab. 21, no.\ 54, 8 pp. (2016)

\bibitem{T00}
Thorisson, H.\: Coupling, Stationarity and Regeneration. Springer, New York (2000)

\bibitem{Z94}
Zhang, Y.\: A note on inhomogeneous percolation.
Ann. Probab. 22, no.\ 2, 803--819. (1994)


\end{thebibliography}
